\documentclass[12pt, a4paper]{amsart}

\usepackage{color}
\usepackage[normalem]{ulem}
\usepackage{tikz, hyperref}
\usetikzlibrary{calc,3d}
\makeatletter
\pgfset{/tikz/cs/radius/.store in=\tikz@cs@radius}
\makeatother
\usepackage[all]{xy}
\usepackage{upref}

\newtheorem{thm}{Theorem}[section]
\newtheorem{prp}[thm]{Proposition}
\newtheorem{cor}[thm]{Corollary}
\newtheorem{lem}[thm]{Lemma}

\newtheorem{prop}[thm]{Proposition}
\newtheorem{obs}[thm]{Observation}

\theoremstyle{definition}
\newtheorem{defn}[thm]{Definition}
\newtheorem{notn}[thm]{Notation}

\theoremstyle{remark}
\newtheorem{rmk}[thm]{Remark}
\newtheorem{example}[thm]{Example}

\newcommand{\inv}{^{-1}}

\def\ceil[#1]{\lceil{#1}\rceil}
\def\floor[#1]{\lfloor{#1}\rfloor}
\def\bceil[#1]{\big\lceil{#1}\big\rceil}
\def\bfloor[#1]{\big\lfloor{#1}\big\rfloor}
\def\Bceil[#1]{\Big\lceil{#1}\Big\rceil}
\def\Bfloor[#1]{\Big\lfloor{#1}\Big\rfloor}

\newcommand{\One}[1]{\mathbf{1}_{#1}}

\newcommand{\NN}{\mathbb{N}}

\newcommand{\RR}{\mathbb{R}}

\newcommand{\ZZ}{\mathbb{Z}}

\newcommand{\tpi}{\wilde{\pi}}
\newcommand{\tphi}{\wilde{\varphi}}
\newcommand{\lt}{\operatorname{lt}}

\newcommand{\Aut}{\operatorname{Aut}}

\newcommand{\tgrphlim}{{
    \renewcommand{\leftarrow}{\leftharpoondown}
    \varprojlim
}}
\numberwithin{equation}{section}

\newcommand{\thmref}[1]{Theorem~\ref{#1}}
\newcommand{\propref}[1]{Proposition~\textup{\ref{#1}}}
\newcommand{\corref}[1]{Corollary~\textup{\ref{#1}}}
\newcommand{\lemref}[1]{Lemma~\textup{\ref{#1}}}
\newcommand{\obsref}[1]{Observation~\textup{\ref{#1}}}

\newcommand{\GG}{\mathcal{G}}

\newcommand{\HH}{\mathcal H}
\newcommand{\NNN}{\mathcal{N}} % this and next one to avoid conflict with Aidan
\newcommand{\CCC}{\mathcal{C}}

\newcommand{\F}{\mathbb{F}}

\newcommand{\R}{\mathbb{R}}
\newcommand{\N}{\mathbb N}

\newcommand{\minus}{\setminus} % \setminus too long to type
\renewcommand{\bar}{\overline} % regular \bar too short
\newcommand{\wilde}{\widetilde} % \tilde too short, \widetilde too long to type

\newcommand{\midtext}[1]{\quad\text{#1}\quad}

\newcommand{\algcov}[1]{\ensuremath{\mathbf{AlgCov}(#1)}}
\newcommand{\topcov}[1]{\ensuremath{\mathbf{TopCov}(X_{#1})}}
\begin{document}
\title[Topological realizations of $k$-graphs]{Topological realizations and fundamental groups of higher-rank graphs}

\author[Kaliszewski]{S. Kaliszewski}
\address{School of Mathematical and Statistical Sciences, Arizona
State University, Tempe, AZ 85287}
\email{kaliszewski@asu.edu}
\author[Kumjian]{Alex Kumjian}
\address{Department of Mathematics (084), University of Nevada, Reno NV 89557-0084}
\email{alex@unr.edu}
\author[Quigg]{John Quigg}
\address{School of Mathematical and Statistical Sciences, Arizona
State University, Tempe, AZ 85287}
\email{quigg@asu.edu}
\author[Sims]{Aidan Sims}
\address{School of Mathematics and Applied Statistics, University of Wollongong, NSW 2522, Australia}
\email{asims@uow.edu.au}

% 05C20: Graph theory; Directed graphs (digraphs), tournaments
% 18D99: Categories with structure; None of the above, but in this section
% 14H30: Coverings, fundamental group


\thanks{This research was supported by the Australian Research Council. Part of the work was completed
while the second author was employed at the University of Wollongong on the ARC grants DP0984339
and DP0984360.}

\date{\today}

\begin{abstract}
We investigate topological realizations of higher-rank graphs. We show that the fundamental group
of a higher-rank graph coincides with the fundamental group of its topological realization. We also
show that topological realization of higher-rank graphs is a functor, and that for each higher-rank
graph $\Lambda$, this functor determines a category equivalence between the category of coverings
of $\Lambda$ and the category of coverings of its topological realization. We discuss how
topological realization relates to two standard constructions for $k$-graphs: projective
limits and crossed products by finitely generated free abelian groups.
\end{abstract}

\maketitle

\section{Introduction}\label{sec:intro}

Higher-rank graphs are higher-dimensional analogues of directed graphs introduced by Kumjian and
Pask in  \cite{KumjianPask:NYJM00}. Their motivation was the study of associated $C^*$-algebras as
common generalizations of the graph $C^*$-algebras of \cite{KumjianPaskEtAl:JFA97} and the
higher-rank Cuntz-Krieger algebras of \cite{RobertsonSteger:JRAM99}.

In \cite{KumjianPask:NYJM00}, Kumjian and Pask described \emph{skew products} of $k$-graphs by
group-valued functors $c$. They showed that if $\Lambda$ is a $k$-graph and $c : \Lambda \to G$ is
a functor into an abelian group, then the $C^*$-algebra associated to the skew-product graph
$\Lambda \times_c G$ is isomorphic to the crossed product of $C^*(\Lambda)$ by an induced action
$\wilde{c}$ of the dual group $\widehat{G}$.

Pask, Quigg and Raeburn extended this result to non-abelian groups \cite{pqr:groupoid, pqr:cover}.
Generalizing results of \cite{DPR} for directed graphs, they showed that if $c : \Lambda \to G$ is
a functor into any discrete group, and $H$ is any subgroup of $G$, then the $C^*$-algebra of the
relative skew product $\Lambda \times_c G/H$ is isomorphic to a  restricted crossed product of
$C^*(\Lambda)$ by a coaction  of $G$. They showed how to interpret relative skew products as
\emph{coverings} of $k$-graphs, and they showed that every covering arises this way by introducing
the fundamental group of a $k$-graph $\Lambda$ and showing that $G$ can be taken to be
$\pi_1(\Lambda)$ and $H$ can be taken such that $H \cong \pi_1(\Lambda \times_c G/H)$. They also
indicated \cite[Section~6]{pqr:groupoid} how one might construct a topological realization of a
$k$-graph by gluing open cells into the interiors of commuting cubes in the category, and indicated
that one would expect the fundamental group of the resulting space to coincide with the fundamental
group of the $k$-graph.

In this paper, we make this precise. We define the topological realization $X_\Lambda$ of a
$k$-graph $\Lambda$ and show by example that a number of standard surfaces arise from this
construction applied to $2$-graphs. We then show that the assignment $\Lambda \to X_\Lambda$
preserves fundamental groups. We go on to show that each $k$-graph morphism $\varphi : \Lambda \to
\Gamma$ induces a continuous map $\wilde{\varphi} : X_\Lambda \to X_\Gamma$ and that the pair
$(\Lambda \mapsto X_\lambda, \varphi \mapsto \wilde{\varphi})$ is a functor from the category of
$k$-graphs with $k$-graph morphisms to the category of topological spaces with continuous maps. The
situation is particularly nice for the coverings studied in \cite{pqr:cover}: for each $k$-graph
$\Lambda$, the assignment $\varphi \mapsto \wilde{\varphi}$ determines a category equivalence
between the category of algebraic coverings of $\Lambda$ and the category of topological coverings
of $X_\Lambda$ that takes a universal covering of $\Lambda$ to  a universal covering of
$X_\Lambda$.

We finish off by describing how our construction behaves with respect to two existing constructions
from the theory of $k$-graphs. Firstly, by analogy with our construction for discrete $k$-graphs,
we propose a notion of topological realization for a topological $k$-graph in the sense of Yeend
\cite{YeendTopGraph}. Given a sequence of finite-to-one coverings $p_n : \Lambda_n \to
\Lambda_{n-1}$ of $k$-graphs, the projective limit $\varprojlim (\Lambda_n, p_n)$ is a topological
$k$-graph \cite{PQS}. We show that the topological realization $X_{\varprojlim (\Lambda_n, p_n)}$
is homeomorphic to the projective limit $\varprojlim(X_{\Lambda_n}, \wilde{p}_n)$, and in
particular that $\pi_1(X_{\varprojlim (\Lambda_n, p_n)}) \cong \varprojlim(\pi_1(\Lambda_n),
(p_n)^*)$. Secondly, we consider the crossed products of $k$-graphs studied in
\cite{FarthingPaskEtAl:HJM09}, and demonstrate that if $\alpha$ is an action of $\ZZ^l$ on a
$k$-graph $\Lambda$, then the topological realization $X_{\Lambda \times_\alpha \ZZ^l}$ of the
crossed-product $k$-graph is homeomorphic to the mapping torus $M(\wilde{\alpha})$ for the induced
homeomorphism $\wilde{\alpha}$ of $X_\Lambda$.

\section{Background}\label{sec:background}

In this paper $\NN$ denotes the natural numbers, which we take to include 0 and regard as a monoid
under addition. For $k \ge 1$ we regard $\NN^k$, the set of $k$-tuples from $\NN$, as a monoid
under pointwise operations. When convenient, we will also regard it as a category with a single
object. We denote the identity element by $0$, and we write $\One{k}$ for the element
$(1,1,\dots,1) \in \NN^k$. We denote the canonical generators of $\NN^k$ by $e_1, \dots, e_k$, and
for $n \in \NN^k$ we write $n_1, \dots, n_k$ for the coordinates of $n$; that is $n = (n_1, n_2,
\dots, n_k) = \sum^k_{i=1} n_i e_i$. We write $|n|$ for $\sum_{i=1}^k n_i$.

For $m,n \in \NN^k$, we write $m \le n$ if $m_i \le n_i$ for all $i$, and $m < n$ if $m \le n$ and
$m \not= n$; in particular, $m < n$ does not mean that $m_i < n_i$ for all $i$. We write $m \vee n$
for the coordinatewise maximum of $m$ and $n$; we then have $m,n \le m \vee n$.

As in \cite{KumjianPask:NYJM00}, a \emph{$k$-graph} is a countable small category $\Lambda$ endowed
with a functor $d : \Lambda \to \NN^k$ satisfying the following factorization property: for all
$\lambda \in \Lambda$ and $m,n \in \NN^k$ such that $d(\lambda) = m+n$ there exist unique elements
$\mu \in d^{-1}(m)$ and $\nu \in d^{-1}(n)$ such that $\lambda = \mu\nu$. We write $\Lambda^n$ for
$d^{-1}(n)$. The map $o \mapsto \operatorname{id}_o$ is a bijection between the objects of
$\Lambda$ and the elements of $\Lambda^0$. We use this to regard the codomain and domain maps on
$\Lambda$ as maps $r,s : \Lambda \to \Lambda^0$, and observe that $\mu$ and $\nu$ are composable if
$s(\mu) = r(\nu)$. We adopt the following notational convention of \cite{pqr:groupoid} for
$k$-graphs. Given $\lambda \in \Lambda$ and $S \subseteq \Lambda$, we write $\lambda S =
\{\lambda\mu : \mu \in S, r(\mu) = s(\lambda)\}$ and $S\lambda = \{\mu\lambda : \mu \in S, s(\mu) =
r(\lambda)\}$. In particular, if $v \in \Lambda^0$ then $vS = r^{-1}(v) \cap S$ and $Sv = s^{-1}(v)
\cap S$.

If $m \le n \le l \in \NN^k$ and $\lambda \in \Lambda^l$, then two applications of the
factorization property show that there exist unique paths $\lambda' \in \Lambda^m$, $\lambda'' \in
\Lambda^{n-m}$ and $\lambda''' \in \Lambda^{l-n}$ such that $\lambda =
\lambda'\lambda''\lambda'''$. We define $\lambda(m, n) = \lambda''$. Since $\lambda =
r(\lambda)\lambda' (\lambda''\lambda''')$ we then have $\lambda(0,m) = \lambda'$ and similarly
$\lambda(n,l) = \lambda'''$.

We emphasize that, while many other papers on $k$-graphs require that $\Lambda$ be finitely-aligned
and/or have no sources, we make no such assumptions in this paper, though many of our key examples
are in fact row-finite.

\section{The topological realization of a higher-rank graph}\label{sec:topological realization}

Let $\Lambda$ be a $k$-graph. Given $t \in \RR^k$, we will write $\ceil[t]$ for the least element
of $\ZZ^k$ which is coordinatewise greater than  or equal to $t$ and $\floor[t]$ for the greatest element of
$\ZZ^k$ which is coordinatewise less than or equal to $t$. Observe that $\floor[t] \le t \le \ceil[t] \le
\floor[t] + \One{k}$ for all $t \in \RR^k$.

Given $p \le q \in \NN^k$, we denote by $[p,q]$ the \emph{closed interval} $\{t \in \RR^k : p \le t
\le q\}$, and we denote by $(p,q)$ the \emph{relatively open interval} $\{t \in [p,q] : p_i < t_i <
q_i\text{ whenever } p_i < q_i\}$. Observe that $(p,q)$ is not open in $\RR^k$ unless $p_i < q_i$
for all $i$, but it is open as a subspace of $[p,q]$. The set $(p,q)$ is never empty: in
particular, if $p = q$, then $(p,q) = [p,q] = \{p\}$. In general, as a subset of $\RR^k$, the
dimension of $(p,q)$ is $|\{i \le k : p_i < q_i\}|$. If $p_i < q_i$ then the
$i$\textsuperscript{th}-coordinate projection of $(p,q)$ is $(p_i, q_i)$, and if $p_i = q_i$ then
the $i$\textsuperscript{th}-coordinate projection of $(p,q)$ is $\{p_i\}$.
\begin{rmk}\label{rmk:box-top}
Let $m \in \NN^k$.  If $m \le \One{k}$, then for all $t \in (0, m)$, $\floor[t]=0$ and $\ceil[t] = m$.
\end{rmk}

We define a relation on the topological disjoint union $\bigsqcup_{\lambda \in \Lambda} \{\lambda\}
\times [0,d(\lambda)]$ by
\begin{equation}\label{eq:equiv rel}
(\mu, s) \sim (\nu, t) \quad\iff\quad \mu(\floor[s], \ceil[s])
    = \nu(\floor[t], \ceil[t])\text{ and } s - \floor[s] = t - \floor[t].
\end{equation}
It is straightforward to see that this is an equivalence relation.

\begin{defn}
Let $\Lambda$ be a $k$-graph. With notation as above, we define the \emph{topological realization}
$X_\Lambda$ of $\Lambda$ to be the quotient space
\[
\Big(\bigsqcup_{\lambda \in \Lambda} \{\lambda\} \times [0, d(\lambda)]\Big)\Big/\sim
\]
\end{defn}

The following alternative characterization of the equivalence relation $\sim$ will simplify
arguments later in the paper.

\begin{lem}\label{lem:sim-plified}
The  relation $\sim$ is generated as an equivalence relation by the relation
\begin{equation}\label{eq:sim generator}
    \big\{\big((\alpha\lambda\beta, t + d(\alpha)), (\lambda,t)\big) : d(\lambda) \le \One{k}\text{ and }t \in [0, d(\lambda)]\big\}.
\end{equation}
\end{lem}
\begin{proof}
The relation~\eqref{eq:sim generator} is contained in $\sim$ by definition of the latter. Now
suppose that $(\mu,s) \sim (\nu,t)$. We must show that $\big((\mu,s), (\nu,t)\big)$ belongs to the
equivalence relation generated by~\eqref{eq:sim generator}. Let
\[
\alpha_\mu = \mu(0, \floor[s]),\quad
        \lambda_\mu = \mu(\floor[s], \ceil[s]), \quad\text{and}\quad
        \beta_\mu = \mu(\ceil[s], d(\mu)),
\]
and similarly for $\nu$.
By definition of $\sim$, we have $(\lambda_\mu, s - \floor[s]) = (\lambda_\nu, t - \floor[t])$.
Since $\big((\mu,s), (\lambda_\mu, s - \floor[s])\big)$ and $\big((\nu,t), (\lambda_\nu, t -
\floor[t])\big)$ belong to~\eqref{eq:sim generator}, $\big((\mu,s), (\nu,t)\big)$ belongs to the
equivalence relation generated by~\eqref{eq:sim generator}.
\end{proof}

\begin{notn}
Let $[\lambda, t]$ denote the equivalence class of an element  $(\lambda, t)$. If $u \in \Lambda^0$
we often write $u$ in place of $[u, 0] \in X_\Lambda$ to simplify notation.

For each $m \le \One{k}$ and each $\lambda \in \Lambda^m$, define
\[
    Q_\lambda = \{[\lambda,t] : t \in (0, m)\} \subset X_\Lambda
  %  \quad\text{and}\quad \overline{Q}_\lambda = \overline{\{[\lambda,t] : t \in (0, m)\}},
\]
and let $\overline{Q}_\lambda$ denote its closure in $X_\Lambda$.
We call $Q_\lambda$ the \emph{open cube} associated to $\lambda$ and $\overline{Q}_\lambda$ the \emph{closed cube} associated to $\Lambda$.
\end{notn}
%Note that if $[\lambda, t] \in Q_\lambda$, then $\floor[t] = 0$ and $\ceil[t] = m$.

\begin{lem}\label{lem:CW structure}
Let $\Lambda$ be a $k$-graph. Then $X_\Lambda = \bigcup_{m \le \One{k}} \bigcup_{\lambda \in
\Lambda^m} Q_\lambda$, and $Q_\lambda \cap Q_\mu = \emptyset$ for distinct
$\lambda,\mu{}\in \bigcup_{m \le \One{k}} \Lambda^m$. For each $m \le \One{k}$ and each
$\lambda \in \Lambda^m$, the map $[\lambda,t] \mapsto t$ is a homeomorphism of $Q_\lambda$ onto
$(0,m) \subset \RR^k$, so $Q_\lambda$ is homeomorphic to the open unit cube in $\RR^{|m|}$. Further
$\overline{Q}_\lambda = \{[\lambda,t] : 0 \le t \le m\}$.

Define $X_\Lambda^0 = \bigcup_{v \in \Lambda^0} Q_v$ and recursively define
\[
X_\Lambda^{r+1} = X_\Lambda^r \cup \Big(\bigcup_{d(\lambda) \le \One{k}, |\lambda| = {r+1}} Q_\lambda\Big).
\]
Then a subset $U$ of $X_\Lambda$ is open if and only if $U \cap X^r_\Lambda$ is relatively open for
each $r \le k$.
\end{lem}
\begin{proof}
Write $Y_\Lambda = \bigsqcup_{\lambda \in \Lambda} \{\lambda\} \times [0, d(\lambda)]$, and let $q
: Y_\Lambda \to X_\Lambda$ be the quotient map.

Fix $[\mu,t] \in X_\Lambda$. Then $\ceil[t] - \floor[t] \le \One{k}$. Moreover, $0 \le t -
\floor[t] \le \ceil[t]-\floor[t]$. Let $\lambda = \mu(\floor[t], \ceil[t])$. Whenever
$\floor[t]_i < \ceil[t]_i$ we have $t_i \not\in \ZZ$ and hence $\floor[t]_i < t_i < \ceil[t]_i$.
So
\[
[\mu,t] = \big[\lambda, t - \floor[t]\big] \in Q_\lambda,
\]
whence $X_\Lambda = \bigcup_{m \le \One{k}} \bigcup_{\lambda \in \Lambda^m} Q_\lambda$.

We now show that the $Q_\lambda$ are mutually disjoint. Fix $\lambda,\mu$ with $0 \le d(\lambda),
d(\mu) \le \One{k}$, and suppose that $[\lambda, s] = [\mu, t] \in Q_\lambda \cap Q_\mu$. We must
show that $\lambda = \mu$. By Remark~\ref{rmk:box-top} $\ceil[s] = d(\lambda)$, $\ceil[t] = d(\mu)$
and $\floor[s] = \floor[t] = 0$. So $s - \floor[s] = t - \floor[t]$ forces $s = t$, and thus
\[
d(\lambda) = \ceil[s] = \ceil[t] = d(\mu).
\]
The definition of $\sim$ then forces
\[
\lambda = \lambda(\floor[s], \ceil[s]) = \mu(\floor[t], \ceil[t]) = \mu.
\]

Fix $\lambda \in \Lambda$ with $d(\lambda) \le \One{k}$.  The above argument also shows that
$[\lambda,t] \mapsto t$ is a well-defined bijection from $Q_\lambda$ onto $(0,m) \subset \RR^k$. So
it remains to check that  the map is a homeomorphism. To see this, observe that if $U$ is
relatively open in $Q_\lambda$, then in particular $\{(\lambda,t) : [\lambda,t] \in U\}$ is open in
$\{\lambda\} \times (0, d(\lambda)) \subseteq \{\lambda\} \times [0, d(\lambda)]$, and hence $\{t :
[\lambda,t] \in U\}$ is open in $(0, d(\lambda))$. So $[\lambda,t] \mapsto t$ is an open map. To
see that it is continuous, fix an open subset $V$ of $(0, d(\lambda))$. Define $W \subseteq
Y_\Lambda$ by
\[\begin{split}
    W = \bigcup\{(\alpha\lambda\beta, t) : s(\alpha) &{}= r(\lambda), r(\beta) = s(\lambda) \\
        &\text{ and }t_i - d(\alpha)_i \in V\text{ whenever } d(\lambda)_i \not= 0\}.
\end{split}\]
Then $W$ is open in $Y_\Lambda$, so $q(W)$ is open in $X_\Lambda$. By definition of $\sim$, we have
$q(W) \cap Q_\lambda = \{[\lambda,t] : t \in V\}$. So $[\lambda,t] \mapsto t$ is continuous as
required.

To see that $\overline{Q}_\lambda = \{[\lambda, t] : t \in [0, d(\lambda)]\}$, we observe that
\[\begin{split}
    q^{-1}(Q_\lambda) \cap (\{\mu\} \times [0, d(\mu)])
        = \{(\mu,t) : {}&\mu(\floor[t], \ceil[t]) = \lambda\\ &\text{ and } d(\lambda)_i \not= 0 \implies t_i \not\in \NN\}.
\end{split}\]
So the closure in $\{\mu\} \times [0, d(\mu)]$ of $q^{-1}(Q_\lambda) \cap (\{\mu\} \times [0,
d(\mu)])$ is
\[
    \{(\mu,t) : \mu(\floor[t], \ceil[t]) = \lambda\},
\]
and the image of this closure under $q$ is precisely $\{[\lambda, t] : t \in [0, d(\lambda)]\}$.

It  remains  to check that $U$ is open in $X_\Lambda$ if and only if $U \cap X^r_\Lambda$ is
relatively open for each $r \le k$. Of course if $U$ is open then each $U \cap X^r_\Lambda$ is
relatively open. On the other hand if each $U \cap X^r_\Lambda$ is relatively open, then in
particular $U = U \cap X_\Lambda = U \cap X^k_\Lambda$ is open.
\end{proof}

\begin{cor}
Let $\Lambda$ be a $k$-graph. Then $X_\Lambda$ is a $k$-dimensional CW-complex with $r$-skeleton
$X^r_\Lambda$ as defined in Lemma~\ref{lem:CW structure} for each $r \le k$. In particular, the
open cells in the $CW$-complex are the $Q_\lambda$ where $d(\lambda) \le \One{k}$, and the closed
cells are the $\overline{Q}_\lambda$.
\end{cor}

In fact, Appendix~A of \cite{KPS3} shows that each $k$-graph $\Lambda$ gives rise to a cubical set
whose $r$-cubes are $\bigcup_{m \le \One{k}, |m| = r} \Lambda^m$, and Theorem~B.2 of \cite{KPS3}
shows that $X_\Lambda$ is homeomorphic to the topological realisation of that cubical set,
providing a somewhat indirect alternative proof of the preceding corollary.

\begin{rmk}
We will mainly be interested in connected $k$-graphs in this paper. However, it is not
difficult to check that $\Lambda$ is connected if and only if $X_\Lambda$ is connected, and
that the connected components of $X_\Lambda$ are precisely the topological realizations of the
connected components of $\Lambda$.
\end{rmk}

\begin{rmk}\label{cwok}
Let $X$ be a connected CW-complex.  Since every CW-complex is locally contractible (see
\cite[Proposition A.4]{Hatcher}), $X$ is path-connected, locally connected and semilocally
simply-connected. Hence, $X$ possesses a universal cover. Let $\Lambda$ be a connected $k$-graph;
then since $X_\Lambda$ is a connected CW-complex it possesses a universal cover as well.
\end{rmk}

\subsection{Examples}

In \cite{KPS3}, a number of examples of $2$-graphs are presented using ``planar diagrams." Here we
will show that the topological realizations of those $2$-graphs are homeomorphic to the spaces
whose homology they were constructed to reflect.

Recall from \cite{KPS3} that the $2$-cubes of a $2$-graph $\Lambda$ are the morphisms of degree
$(1,1)$, and that a commuting diagram (in the category $\Lambda$) that includes all $2$-cubes as
commuting squares is called a \emph{planar diagram} for $\Lambda$. To see how these diagrams relate
to the topological realizations of the corresponding $2$-graphs, we present another description of
$X_\Lambda$ in terms of cubes.

\begin{lem}
Let $\Lambda$ be a $k$-graph. Then $X_\Lambda$ is homeomorphic to the quotient of the topological
disjoint union $\bigsqcup_{d(\lambda) \le \One{k}} \{\lambda\} \times [0, d(\lambda)]$ by the
equivalence relation $R$ generated by
\begin{align*}\textstyle
\bigcup_{m \le \One{k}} \bigcup_{m_i = 1}
    \Big(\big\{\big(&\textstyle(\lambda\alpha,t), (\lambda,t)\big) :
            \lambda \in \Lambda^{m-e_i}, \alpha \in s(\lambda)\Lambda^{e_i}\big\} \\
        &\textstyle{}\cup \big\{\big((\alpha\lambda, t+d(\alpha)), (\lambda,t)\big) :
            \lambda \in \Lambda^{m-e_i}, \alpha \in \Lambda^{e_i} r(\lambda)\}\Big).
\end{align*}
\end{lem}
\begin{proof}
By \cite[Theorem~B.2]{KPS3}, $X_\Lambda$ is homeomorphic to the topological realization of the
associated cubical set. Since, in the topological realization of a cubical set, every point has a
representative in a nondegenerate cube, the topological realization of the cubical set of $\Lambda$
is precisely the quotient described above.
\end{proof}

The preceding lemma implies that if $E$ is a planar diagram for a $2$-graph $\Lambda$, then the
topological realization of $\Lambda$ is homeomorphic to the space obtained by pasting a unit square
into each commuting square in $E$, and then identifying all instances of any given edge or vertex
in an orientation-preserving way.

To describe the examples in this section, recall that the $1$-skeleton, or just skeleton, of a
$k$-graph is the directed graph $E$ with vertices $\Lambda^0$ and edges $\bigsqcup^k_{i=1}
\Lambda^{e_i}$ drawn using $k$-different colours to distinguish the different degrees. There is a
complete characterisation of which $k$-colored graphs give rise to $k$-graphs
\cite{FowlerSims:TAMS02, HazelwoodRaeburnEtAl:xx11}, but for our purposes it suffices to recall the
following special case of the construction of \cite[Section~6]{KumjianPask:NYJM00}: if $E$ is a
$2$-colored graph (with edges colored blue and red, say) and if, for every bi-colored path $ef$ in
$E^2$ there is a unique bi-colored $f'e'$ with the same range and source but with the colors
occurring in the reverse order, then there is a unique $2$-graph $\Lambda$ whose $1$-skeleton is
$E$.

The diagrams in the following examples are reproduced from \cite{KPS3}; edges of degree $(1,0)$ are
blue and solid, and edges of degree $(0,1)$ are red and dashed.

\begin{example}[{\cite[Example~5.4]{KPS3}}]\label{ex:sphere}
Let $\Lambda$ be the $2$-graph with planar diagram and skeleton below.
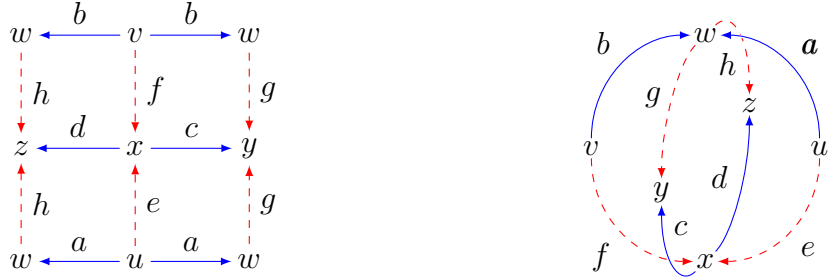
\[
\begin{tikzpicture}[scale=1.5]
    \node[inner sep=1pt, circle] (sw) at (-1,-1) {$w$};
    \node[inner sep=1pt, circle] (w) at (-1,0) {$z$};
    \node[inner sep=1pt, circle] (nw) at (-1,1) {$w$};
    \node[inner sep=1pt, circle] (s) at (0,-1) {$u$};
    \node[inner sep=1pt, circle] (m) at (0,0) {$x$};
    \node[inner sep=1pt, circle] (n) at (0,1) {$v$};
    \node[inner sep=1pt, circle] (se) at (1,-1) {$w$};
    \node[inner sep=1pt, circle] (e) at (1,0) {$y$};
    \node[inner sep=1pt, circle] (ne) at (1,1) {$w$};
    \draw[-latex, blue] (s)--(sw) node[pos=0.5, above] {\color{black}$a$};
    \draw[-latex, blue] (s)--(se) node[pos=0.5, above] {\color{black}$a$};
    \draw[-latex, blue] (m)--(w) node[pos=0.5, above] {\color{black}$d$};
    \draw[-latex, blue] (m)--(e) node[pos=0.5, above] {\color{black}$c$};
    \draw[-latex, blue] (n)--(nw) node[pos=0.5, above] {\color{black}$b$};
    \draw[-latex, blue] (n)--(ne) node[pos=0.5, above] {\color{black}$b$};
    \draw[-latex, red, dashed] (sw)--(w) node[pos=0.5, right] {\color{black}$h$};
    \draw[-latex, red, dashed] (nw)--(w) node[pos=0.5, right] {\color{black}$h$};
    \draw[-latex, red, dashed] (s)--(m) node[pos=0.5, right] {\color{black}$e$};
    \draw[-latex, red, dashed] (n)--(m) node[pos=0.5, right] {\color{black}$f$};
    \draw[-latex, red, dashed] (se)--(e) node[pos=0.5, right] {\color{black}$g$};
    \draw[-latex, red, dashed] (ne)--(e) node[pos=0.5, right] {\color{black}$g$};
%
%    \node at (-0.42, -0.42) {$\beta$};
%    \node at (-0.42, 0.58) {$\gamma$};
%    \node at (0.58, -0.42) {$\alpha$};
%    \node at (0.58, 0.58) {$\delta$};
%
\begin{scope}[xshift=5cm]
    \node[inner sep= 1pt] (100) at (1,0,0) {$u$};
    \node[inner sep= 1pt] (-100) at (-1,0,0) {$v$};
    \node[inner sep= 1pt] (010) at (0,1,0) {$w$};
    \node[inner sep= 1pt] (0-10) at (0,-1,0) {$x$};
    \node[inner sep= 1pt] (001) at (0,0,1) {$y$};
    \node[inner sep= 1pt] (00-1) at (0,0,-1) {$z$};
    \draw[-latex, red, dashed] (010) .. controls +(0,0,0.6) and +(0,0.6,0) .. (001) node[pos=0.5, anchor=east] {\color{black}$g$};
    \draw[-latex, red, dashed] (010) .. controls +(0,0,-0.6) and +(0,0.6,0) .. (00-1) node[pos=0.8, anchor=east] {\color{black}$h$};
    \draw[-latex, blue] (0-10) .. controls +(0,0,0.6) and +(0,-0.6,0) .. (001) node[pos=0.85, anchor=west] {\color{black}$c$};
    \draw[-latex, blue] (0-10) .. controls +(0,0,-0.6) and +(0,-0.6,0) .. (00-1) node[pos=0.5, anchor=south east] {\color{black}$d$};
    \draw[-stealth, white, very thick] (100) .. controls +(0,0.6,0) and +(0.6,0,0) .. (010) node[pos=0.5, anchor=south west] {\color{black}$a$};
    \draw[-latex, blue] (100) .. controls +(0,0.6,0) and +(0.6,0,0) .. (010) node[pos=0.5, anchor=south west] {\color{black}$a$};
    \draw[-latex, red, dashed] (100) .. controls +(0,-0.6,0) and +(0.6,0,0) .. (0-10) node[pos=0.5, anchor=north west] {\color{black}$e$};
    \draw[-latex, blue] (-100) .. controls +(0,0.6,0) and +(-0.6,0,0) .. (010) node[pos=0.5, anchor=south east] {\color{black}$b$};
    \draw[-latex, red, dashed] (-100) .. controls +(0,-0.6,0) and +(-0.6,0,0) .. (0-10) node[pos=0.5, anchor=north east] {\color{black}$f$};
\end{scope}
\end{tikzpicture}
\]
If we paste a square into each commuting square in the planar diagram on the left, and then
identify all instances of any given edge or vertex, the resulting space is that obtained by pasting
a square onto each commuting square in the skeleton on the right, so is homeomorphic to a sphere.
In particular, the fundamental group of this $2$-graph is trivial.
\end{example}

\begin{example}[{\cite[Example~5.5]{KPS3}}]
Consider the $2$-graph $\Sigma$ with planar diagram on the left and skeleton on the right in the
following diagram.
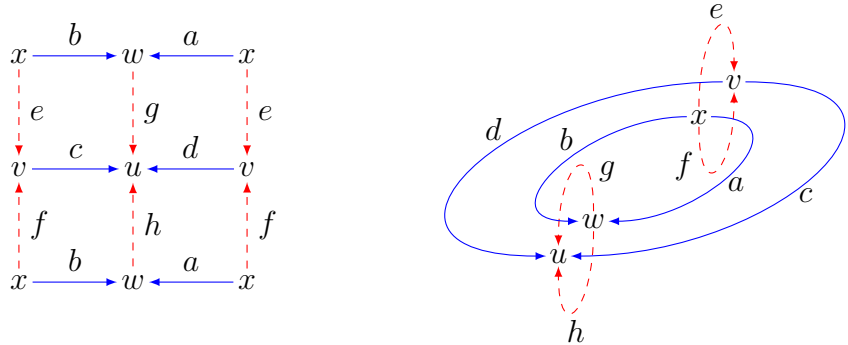
\[
\begin{tikzpicture}[scale=1.5]
    \node at (-2,0) {};
    \node[inner sep=0.5pt, circle] (sw) at (-1,-1) {$x$};
    \node[inner sep=0.5pt, circle] (w) at (-1,0) {$v$};
    \node[inner sep=0.5pt, circle] (nw) at (-1,1) {$x$};
    \node[inner sep=0.5pt, circle] (s) at (0,-1) {$w$};
    \node[inner sep=0.5pt, circle] (m) at (0,0) {$u$};
    \node[inner sep=0.5pt, circle] (n) at (0,1) {$w$};
    \node[inner sep=0.5pt, circle] (se) at (1,-1) {$x$};
    \node[inner sep=0.5pt, circle] (e) at (1,0) {$v$};
    \node[inner sep=0.5pt, circle] (ne) at (1,1) {$x$};
    \draw[-latex, blue] (sw)--(s) node[pos=0.5, above, black] {$b$};
    \draw[-latex, blue] (se)--(s) node[pos=0.5, above, black] {$a$};
    \draw[-latex, blue] (w)--(m) node[pos=0.5, above, black] {$c$};
    \draw[-latex, blue] (e)--(m) node[pos=0.5, above, black] {$d$};
    \draw[-latex, blue] (nw)--(n) node[pos=0.5, above, black] {$b$};
    \draw[-latex, blue] (ne)--(n) node[pos=0.5, above, black] {$a$};
    \draw[-latex, red, dashed] (sw)--(w) node[pos=0.5, right, black] {$f$};
    \draw[-latex, red, dashed] (nw)--(w) node[pos=0.5, right, black] {$e$};
    \draw[-latex, red, dashed] (s)--(m) node[pos=0.5, right, black] {$h$};
    \draw[-latex, red, dashed] (n)--(m) node[pos=0.5, right, black] {$g$};
    \draw[-latex, red, dashed] (se)--(e) node[pos=0.5, right, black] {$f$};
    \draw[-latex, red, dashed] (ne)--(e) node[pos=0.5, right, black] {$e$};
\begin{scope}[xshift=4.5cm]
    \node[inner sep= 1pt] (002) at (0,0,2) {$u$};
    \node[inner sep= 1pt] (001) at (0,0,1.2) {$w$};
    \node[inner sep= 1pt] (00-1) at (0,0,-1.2) {$x$};
    \node[inner sep= 1pt] (00-2) at (0,0,-2) {$v$};
    \draw[-latex, red, dashed] (001) .. controls +(0,0.8,0) and +(0,0.8,0) .. (002) node[pos=0.25, anchor=west] {\color{black}$g$};
    \draw[-latex, red, dashed] (001) .. controls +(0,-0.8,0) and +(0,-0.8,0) .. (002) node[pos=0.5, anchor=north] {\color{black}$h$};
    \draw[-latex, red, dashed] (00-1) .. controls +(0,0.8,0) and +(0,0.8,0) .. (00-2) node[pos=0.5, anchor=south] {\color{black}$e$};
    \draw[-latex, red, dashed] (00-1) .. controls +(0,-0.8,0) and +(0,-0.8,0) .. (00-2) node[pos=0.25, anchor=east] {\color{black}$f$};
    \draw[-latex, blue] (00-1) .. controls +(1,0,0) and +(1,0,0) .. (001) node[pos=0.6, anchor=west] {\color{black}$a$};
    \draw[-latex, blue] (00-1) .. controls +(-1,0,0) and +(-1,0,0) .. (001) node[pos=0.4, inner sep=1pt, anchor=south east] {\color{black}$b$};
    \draw[-latex, blue] (00-2) .. controls +(2,0,0) and +(2,0,0) .. (002) node[pos=0.6, anchor=west] {\color{black}$c$};
    \draw[-latex, blue] (00-2) .. controls +(-2,0,0) and +(-2,0,0) .. (002) node[pos=0.4, inner sep=1pt, anchor=south east] {\color{black}$d$};
\end{scope}
\end{tikzpicture}
\]
The argument of the preceding example shows that the topological realization of this $2$-graph is a
$2$-torus. In particular, its fundamental group is $\ZZ^2$.
\end{example}

\begin{example}[{\cite[Example~5.6]{KPS3}}]\label{eg:projective}
Let $\Lambda$ be the $2$-graph with planar diagram on the left and skeleton on the right in the
following diagram.
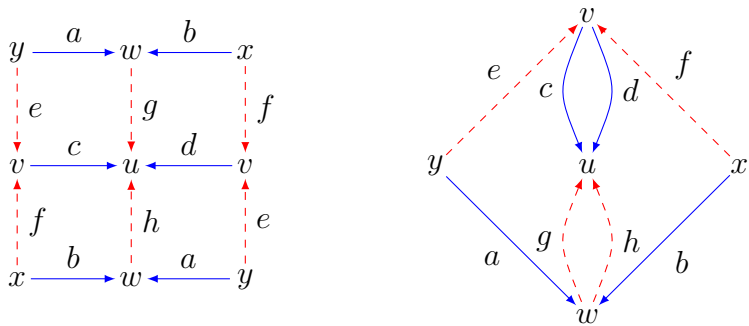
\[
\begin{tikzpicture}
\begin{scope}[scale=1.5]
    \node[inner sep=0.5pt, circle] (sw) at (-1,-1) {$x$};
    \node[inner sep=0.5pt, circle] (w) at (-1,0) {$v$};
    \node[inner sep=0.5pt, circle] (nw) at (-1,1) {$y$};
    \node[inner sep=0.5pt, circle] (s) at (0,-1) {$w$};
    \node[inner sep=0.5pt, circle] (m) at (0,0) {$u$};
    \node[inner sep=0.5pt, circle] (n) at (0,1) {$w$};
    \node[inner sep=0.5pt, circle] (se) at (1,-1) {$y$};
    \node[inner sep=0.5pt, circle] (e) at (1,0) {$v$};
    \node[inner sep=0.5pt, circle] (ne) at (1,1) {$x$};
    \draw[-latex, blue] (sw)--(s) node[pos=0.5, above, black] {$b$};
    \draw[-latex, blue] (se)--(s) node[pos=0.5, above, black] {$a$};
    \draw[-latex, blue] (w)--(m) node[pos=0.5, above, black] {$c$};
    \draw[-latex, blue] (e)--(m) node[pos=0.5, above, black] {$d$};
    \draw[-latex, blue] (nw)--(n) node[pos=0.5, above, black] {$a$};
    \draw[-latex, blue] (ne)--(n) node[pos=0.5, above, black] {$b$};
    \draw[-latex, red, dashed] (sw)--(w) node[pos=0.5, right, black] {$f$};
    \draw[-latex, red, dashed] (nw)--(w) node[pos=0.5, right, black] {$e$};
    \draw[-latex, red, dashed] (s)--(m) node[pos=0.5, right, black] {$h$};
    \draw[-latex, red, dashed] (n)--(m) node[pos=0.5, right, black] {$g$};
    \draw[-latex, red, dashed] (se)--(e) node[pos=0.5, right, black] {$e$};
    \draw[-latex, red, dashed] (ne)--(e) node[pos=0.5, right, black] {$f$};
%
%    \node at (-0.5, -0.5) {$\gamma$};
%    \node at (-0.5, 0.5) {$\alpha$};
%    \node at (0.5, -0.5) {$\beta$};
%    \node at (0.5, 0.5) {$\delta$};
\end{scope}
\begin{scope}[scale=2, xshift=3cm]
    \node[inner sep=.5pt, circle] (u) at (0,0) {$u$};
    \node[inner sep=.5pt, circle] (v) at (0,1) {$v$};
    \node[inner sep=.5pt, circle] (w) at (0,-1) {$w$};
    \node[inner sep=.5pt, circle] (x) at (1,0) {$x$};
    \node[inner sep=.5pt, circle] (y) at (-1,0) {$y$};
    \draw[-latex, blue] (v) .. controls +(-0.2,-0.5) .. (u) node[pos=0.5, left, black] {$c$};
    \draw[-latex, blue] (v) .. controls +(0.2,-0.5) .. (u) node[pos=0.5, right, black] {$d$};
    \draw[-latex, red, dashed] (w) .. controls +(-0.2,0.5) .. (u) node[pos=0.5, left, black] {$g$};
    \draw[-latex, red, dashed] (w) .. controls +(0.2,0.5) .. (u) node[pos=0.5, right, black] {$h$};
    \draw[-latex, blue] (x)--(w) node[pos=0.5, anchor=north west, black] {$b$};
    \draw[-latex, red, dashed] (x)--(v) node[pos=0.5, anchor=south west, black] {$f$};
    \draw[-latex, blue] (y)--(w) node[pos=0.5, anchor=north east, black] {$a$};
    \draw[-latex, red, dashed] (y)--(v) node[pos=0.5, anchor=south east, black] {$e$};
\end{scope}
\end{tikzpicture}
\]
Arguing as in the preceding two examples, we see that the topological realization of this $2$-graph
is homeomorphic to the projective plane.
\end{example}

\begin{example}\label{ex:Klein}
Consider the $2$-graph $\Lambda$ with planar diagram on the left and skeleton on the right in the
following diagram.
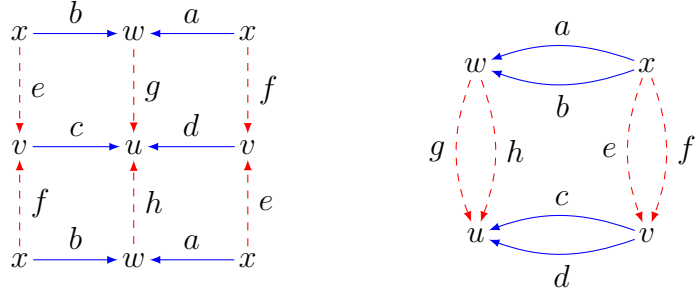
\[
\begin{tikzpicture}[scale=1.5]
    \node[inner sep=0.5pt, circle] (sw) at (-1,-1) {$x$};
    \node[inner sep=0.5pt, circle] (w) at (-1,0) {$v$};
    \node[inner sep=0.5pt, circle] (nw) at (-1,1) {$x$};
    \node[inner sep=0.5pt, circle] (s) at (0,-1) {$w$};
    \node[inner sep=0.5pt, circle] (m) at (0,0) {$u$};
    \node[inner sep=0.5pt, circle] (n) at (0,1) {$w$};
    \node[inner sep=0.5pt, circle] (se) at (1,-1) {$x$};
    \node[inner sep=0.5pt, circle] (e) at (1,0) {$v$};
    \node[inner sep=0.5pt, circle] (ne) at (1,1) {$x$};
    \draw[-latex, blue] (sw)--(s) node[pos=0.5, above, black] {$b$};
    \draw[-latex, blue] (se)--(s) node[pos=0.5, above, black] {$a$};
    \draw[-latex, blue] (w)--(m) node[pos=0.5, above, black] {$c$};
    \draw[-latex, blue] (e)--(m) node[pos=0.5, above, black] {$d$};
    \draw[-latex, blue] (nw)--(n) node[pos=0.5, above, black] {$b$};
    \draw[-latex, blue] (ne)--(n) node[pos=0.5, above, black] {$a$};
    \draw[-latex, red, dashed] (sw)--(w) node[pos=0.5, right, black] {$f$};
    \draw[-latex, red, dashed] (nw)--(w) node[pos=0.5, right, black] {$e$};
    \draw[-latex, red, dashed] (s)--(m) node[pos=0.5, right, black] {$h$};
    \draw[-latex, red, dashed] (n)--(m) node[pos=0.5, right, black] {$g$};
    \draw[-latex, red, dashed] (se)--(e) node[pos=0.5, right, black] {$e$};
    \draw[-latex, red, dashed] (ne)--(e) node[pos=0.5, right, black] {$f$};
%
%    \node at (-0.5, -0.5) {$\gamma$};
%    \node at (-0.5, 0.5) {$\alpha$};
%    \node at (0.5, -0.5) {$\beta$};
%    \node at (0.5, 0.5) {$\delta$};
%
\begin{scope}[xshift=3cm, yshift=-0.78cm, scale=1.5]
    \node[inner sep=0.5pt, circle] (u) at (0,0) {$u$};
    \node[inner sep=0.5pt, circle] (v) at (1,0) {$v$};
    \node[inner sep=0.5pt, circle] (w) at (0,1) {$w$};
    \node[inner sep=0.5pt, circle] (x) at (1,1) {$x$};
    \draw[blue,-latex,out=160, in=20] (v) to node[pos=0.5,above, black] {$c$} (u);
    \draw[blue,-latex,out=200, in=340] (v) to node[pos=0.5,below, black] {$d$} (u);
    \draw[blue,-latex,out=160, in=20] (x) to node[pos=0.5,above, black] {$a$} (w);
    \draw[blue,-latex,out=200, in=340] (x) to node[pos=0.5,below, black] {$b$} (w);
    \draw[red, dashed, -latex, out=290, in=70] (w) to node[pos=0.5,right, black] {$h$} (u);
    \draw[red, dashed, -latex, out=250, in=110] (w) to node[pos=0.5,left, black] {$g$} (u);
    \draw[red, dashed, -latex, out=290, in=70] (x) to node[pos=0.5,right, black] {$f$} (v);
    \draw[red, dashed, -latex, out=250, in=110] (x) to node[pos=0.5,left, black] {$e$} (v);
\end{scope}
\end{tikzpicture}
\]
Arguing as above, we see that the topological realization of this $2$-graph is a Klein bottle, and
in particular that its fundamental group is $\F_2/\langle abab^{-1}\rangle$.
\end{example}

\section{The fundamental group of a higher-rank graph}\label{sec:fundamental group}

In proving that the algebraic and topological fundamental groups of a $k$-graph are isomorphic, on
the algebraic side we need to pass from the fundamental groupoid to the fundamental group. Here we
show how to do this.

First of all, we quote the following theorems from topology (see, e.g., \cite{massey}). Let $X$ be
a connected $k$-dimensional CW-complex.

\begin{enumerate}
\item \cite[Theorem~VII.4.1]{massey} The inclusion $X^2\hookrightarrow X$ induces an
    isomorphism $\pi_1(X^2)\cong \pi_1(X)$.

\item \cite[Theorem~VII.2.1]{massey} Let $\iota:X^1\hookrightarrow X^2$ denote the inclusion
    map. Denote by $Q = ((0,0),(1,1))$ the open unit square in $\R^2$, let $\bar Q$ be the
    closed unit square, and let $\partial Q=\bar Q\minus Q$. Each 2-cell attached to form $X^2$
    is determined by a ``characteristic map'' $f_i:\bar Q\to X^2$, namely a continuous map
    taking $Q$ homeomorphically onto an open set in $X^2\minus X^1$ such that $f_i(\partial
    Q)\subset X^1\minus X^2$. Let $\varphi$ denote a generator of $\pi_1(\partial Q)$. Then
    $\iota_*:\pi_1(X^1)\to \pi_1(X^2)$ is a surjective homomorphism whose kernel is the normal
    subgroup generated by the images $f_i{}_*(\varphi)$ under the characteristic maps.

\item \cite[Theorem~VI.5.2]{massey}
$\pi_1(X^1)\cong \F_n$, where $n$ is the cardinality of the set of edges in $E^1$ remaining after a maximal tree has been removed.
\end{enumerate}

We next give some background from \cite{schubert}.
Let $\CCC$ be a small category.

A \emph{congruence relation} on $\CCC$ is an equivalence relation $R$ on $\CCC$ such that
\begin{align}
\label{fibre}
&\text{if $(\alpha,\beta)\in R$ then $s(\alpha)=s(\beta)$ and $r(\alpha)=r(\beta)$, and}
\\
\label{compose}
&\text{if $(\alpha,\beta),(\lambda,\mu)\in R$ and $s(\alpha)=r(\lambda)$, then
$(\alpha\lambda,\beta\mu)\in R$.}
\end{align}
In this case the quotient $\CCC/R$ is a category, and the quotient map $Q:\CCC\to \CCC/R$ is a functor.

If $S\subset \CCC\times\CCC$ satisfies \eqref{fibre}, then there is a smallest congruence relation
on $\CCC$ containing $S$, which we say is \emph{generated} by $S$.

We are primarily interested in the case where $\CCC$ is a groupoid, which we will typically denote
by $\GG$. Also, we will make the standing assumption that $\GG$ is connected in the sense
that $v \GG u \not= \emptyset$ for all units $v,u$ of $\GG$.

A subgroupoid $\NNN$ of $\GG$ is \emph{normal} if
\begin{align}
&\label{wide}\text{$\NNN^0=\GG^0$, and}
\\
&\label{conjugate}\text{$\beta\alpha\beta\inv\in \NNN$ for all
$\alpha\in\NNN(u)$,
$\beta\in \GG u$,
and
$u\in\GG^0$.}
\end{align}

The following is our main technical tool allowing us to pass from the fundamental groupoid to the
fundamental group. Recall that for a unit $u$ of a groupoid $\GG$, we write $\GG(u)$ for
the isotropy group $u\GG u$ at $u$.

\begin{prop}\label{technical}
Let $S\subset\GG\times\GG$ satisfy \eqref{fibre}, let $R$ be the congruence relation on $\GG$
generated by $S$, let $\HH=\GG/R$ be the quotient groupoid, and let $Q:\GG\to\HH$ be the quotient
map. Fix $u\in\GG^0$, and for each $v\in\GG^0$ choose $\kappa_v\in v\GG u$, with $\kappa_u=u$. Let
$K$ be the normal subgroup of $\GG(u)$ generated by
\begin{equation}\label{connect squares}
\{\kappa_{r(\alpha)}\inv\alpha\beta\inv\kappa_{r(\alpha)}:(\alpha,\beta)\in S\}.
\end{equation}
Then $\HH(u)=\GG(u)/K$.
\end{prop}

\begin{proof}
Let $\NNN=\ker Q$, so that $\HH=\GG/\NNN$.
Put
\[
T=\{\alpha\beta\inv:(\alpha,\beta)\in S\}.
\]
Then $\NNN$ is the normal subgroupoid of $\GG$ generated by $T$, and $K$ is the normal subgroup of
$\GG(u)$ generated by
\begin{equation}\label{sufficient set}
\bigcup_{v\in\GG^0}\kappa_v\inv(T\cap\GG(v))\kappa_v.
\end{equation}
It suffices to show that $\NNN(u)=K$.

Since $T\subset\bigcup_{v\in\GG^0}\GG(v)$, the group $\NNN(u)$ coincides with the subgroup of
$\GG(u)$ generated by
\[
\bigcup_{\substack{v\in\GG^0\\\beta\in v\GG u}}\beta\inv(T\cap\GG(v))\beta.
\]
We need to know that $\NNN(u)$ is generated as a normal subgroup by the smaller set
\eqref{sufficient set}. Since $K$ is normal, it is easy to see that for each $v\in\GG^0$ and
$\beta\in v\GG u$ we have
\[
\beta\inv(T\cap\GG(v))\beta\subset\kappa_v\inv(T\cap\GG(v))\kappa_v,
\]
and the result follows.
\end{proof}

Our next goal is to show how to pass from the fundamental groupoid to the fundamental
group, which we do in Corollary~\ref{isomorphism}. Let $\Lambda$ be a connected $k$-graph, and let
$E$ be its 1-skeleton. Let $\GG(\Lambda)$ and $\GG(E)$ denote the fundamental groupoids of
$\Lambda$ and $E$, respectively. We will find it convenient to package the commuting squares of
$\Lambda$ as ``commutativity conditions''\footnote{in the terminology of \cite{schubert}} in $E$:
each commuting square is of the form $\lambda=ef=gh$ with $d(e)=d(h)=e_i$, $e(f)=d(g)=e_j$, and
$i\ne j$. We regard the edge-paths $ef$ and $gh$ as elements of $\GG(E)$, we associate to $\lambda$
the pair $(ef,gh)\in \GG(E)\times\GG(E)$, and we let $S$ denote the set of all such pairs (and we
abuse terminology by referring to these pairs as commuting squares also).

As discussed in the paragraph following \cite[Definition~5.6]{pqr:groupoid}, it follows from \cite[Theorem~5.5]{pqr:groupoid} that
\[
\GG(\Lambda)\cong \GG(E)/R,
\]
where $R$ is the congruence relation on $\GG(E)$ generated by $S$.
In the following corollary, we identify the fundamental group;
this corollary can be regarded as making precise the discussion in \cite[Section~6]{pqr:groupoid}.

\begin{cor}\label{pass}
Let $S\subset \GG(E)\times \GG(E)$ be the commuting squares of a connected $k$-graph $\Lambda$,
and let $R$ be the congruence relation on $\GG(E)$ generated by $S$.
Then for any vertex $u\in\Lambda^0$ we have
\[
\pi_1(\Lambda,u)\cong \pi_1(E,u)/K,
\]
where $K$ is the normal subgroup of $\pi(E,u)$ generated by the set
\[
\{\kappa_{r(\alpha)}\inv\alpha\beta\inv\kappa_{r(\alpha)}:(\alpha,\beta)\in S\}.
\]
\end{cor}

\begin{proof}
This follows immediately from \propref{technical}.
\end{proof}

We now proceed toward our main result on the fundamental groups, namely
$\pi_1(\Lambda,u)\cong \pi_1(X_\Lambda,u)$.

First, some notation: for each $n\ge 1$ and each commuting $n$-cube $\lambda\in\Lambda$, let
$f_\lambda$ be the associated map attaching an $n$-cell to $X^{n-1}$ in the formation of $X^n$. Let
$Q^n$ be the open unit cube in $\R^n$. Recall that
\[
Q_\lambda=f_\lambda(Q^n)\quad\text{and}\quad
\bar Q_\lambda=f_\lambda(\bar{Q^n}).
\]
Moreover, if $e \in E^1$ then the homotopy class $[f_e]$ may be regarded as an element of
$\GG(X^1)$, the fundamental groupoid of the $1$-skeleton.

\begin{lem}\label{1 isomorphism}
Define $\theta:E^1\to\GG(X^1)$ by
\[
\theta(e)=[f_e].
%\quad\text{for all $e\in E^1$}.
\]
Then $\theta$ extends to a groupoid homomorphism of $\GG(E)$ into $\GG(X^1)$, and for each
$u\in\Lambda^0$, $\theta$ restricts to an isomorphism $\theta:\pi_1(E,u)\to \pi_1(X^1,u)$.
\end{lem}

\begin{proof}
This is routine, although the result in this form does not appear to be readily available in the literature.
The first part follows from the techniques in the proof of \cite[Theorem~3.7.3]{spanier}, and then the second part follows from the observations that,
since $\Lambda$ is connected, so are the 1-skeleton $E$ and the 1-dimensional CW-complex $X^1$, and
both fundamental groups $\pi_1(E,u)$ and $\pi_1(X^1,u)$ are free, with the same number of generators (see, e.g., \cite[Corollary~3.7.5]{spanier}, \cite[Theorem~6.5.2]{massey}).
\end{proof}

For the following lemma, recall from the opening of the section that $\varphi$ denotes the boundary
of the unit square in $\R^2$.

\begin{lem}\label{normal}
Let $K$ be as in \corref{pass}, and let $L$ be the normal subgroup of $\pi_1(X^1,u)$ generated by
$\{f_\lambda(\varphi):\text{$\lambda$ is a commuting square in $\Lambda$}\}$. Then $\theta(K)=L$.
\end{lem}

\begin{proof}
Let $\lambda=ef=gh$, where $d(e)=d(h)=e_i)$ and $d(f)=d(g)=e_j$ with $i\ne j$.
Then it follows from the definitions that
\[
\theta\bigl(\kappa_{r(e)}\inv efh\inv g\inv\kappa_{r(e)}\bigr)
\quad\text{is equal either to}\quad f_\lambda(\varphi)\quad\text{or to}\quad f_\lambda(\varphi\inv),
\]
and the lemma follows.
\end{proof}

\begin{cor}\label{isomorphism}
The isomorphism $\pi_1(E,u)\cong \pi_1(X^1,u)$ of fundamental groups of $1$-skeletons induces an
isomorphism $\pi_1(\Lambda,u)\cong \pi_1(X_\Lambda,u)$.
\end{cor}

\begin{proof}
This follows from Lemmas~\ref{1 isomorphism} and \ref{normal} because
\[
\pi_1(\Lambda,u)\cong \pi_1(E,u)/K
\quad\text{and}\quad
\pi_1(X_\Lambda,u)\cong \pi_1(X^1,u)/L.
\qedhere
\]
\end{proof}

\section{Functoriality}\label{sec:functoriality}

We prove here that quasimorphisms of $k$-graphs induce continuous maps of topological realizations.
In particular, topological realization is a functor from the category of higher rank graphs and
quasimorphisms to that of topological spaces and continuous maps. For quasimorphisms which carry
edges to edges --- for example, $k$-graph morphisms --- the induced map of topological realizations
is injective if and only if the original quasimorphism is injective, and is surjective if and only
if the original quasimorphism is surjective.

Recall from \cite{KPS3} that if $\pi : \NN^k \to \NN^l$ is a homomorphism, and if $\Lambda$ is a
$k$-graph and $\Gamma$ an $l$-graph, then a $\pi$-quasimorphism from $\Lambda$ to $\Gamma$ is a
functor $\phi : \Lambda \to \Gamma$ such that $d(\phi(\lambda)) = \pi(d(\lambda))$ for all $\lambda
\in \Lambda$.

\begin{prp}\label{prp:functorial}
Let $\Lambda$ be a $k$-graph and $\Gamma$ an $l$-graph. Fix a homomorphism $\pi : \NN^k \to \NN^l$.
Extend this to a homomorphism $\tpi : \RR^k \to \RR^l$ by $\tpi(t) = \sum^k_{i=1} t_i \pi(e_i)$.
Suppose that $\varphi : \Lambda \to \Gamma$ is a $\pi$-quasimorphism. Then there is a continuous
map $\tphi : X_\Lambda \to X_\Gamma$ defined by
\[
\tphi([\lambda, t]) = [\varphi(\lambda), \tpi(t)].
\]
Moreover, if $\pi' : \NN^l \to \NN^h$ is another homomorphism, $\Sigma$ is an $h$-graph and
$\varphi' : \Gamma \to \Sigma$ is a $\pi'$-quasimorphism, then $\varphi' \circ \varphi$ is a $\pi'
\circ \pi$-quasimorphism, and $\tphi' \circ \tphi = (\varphi \circ \varphi')^{\sim}$.
\end{prp}
\begin{proof}
We first show that $\tphi$ is well-defined. Suppose that $(\lambda, s) \sim (\mu,t)$. We must show
that $(\varphi(\lambda), \tpi(s)) \sim (\varphi(\mu), \tpi(t))$. We have $\tpi(t) = \tpi(\floor[t])
+ \tpi(t - \floor[t])$. Since $\tpi(\floor[t]) = \pi(\floor[t]) \in \NN^l$, we have
$\bfloor[{\tpi(\floor[t]) + x}] = \pi(\floor[t]) + \floor[\tpi(x)]$ for all $x \in \RR^k$. Hence
\begin{align*}
\tpi(s) - \floor[\tpi(s)]
    &= \big(\tpi(\floor[s]) + \tpi(s - \floor[s])\big) - \bfloor[{\tpi(\floor[s]) + \tpi(s- \floor[s])}] \\
    &= \tpi(\floor[s]) + \tpi(s - \floor[s]) - \tpi(\floor[s]) - \bfloor[{\tpi(s - \floor[s])}] \\
    &= \tpi(s - \floor[s]) - \bfloor[{\tpi(s - \floor[s])}].
\end{align*}
Likewise, $\tpi(t) - \floor[\tpi(t)] = \tpi(t - \floor[t]) - \bfloor[{\tpi(t - \floor[t])}]$. Since
$(\lambda, s) \sim (\mu, t)$, we have $s - \floor[s] = t - \floor[t]$, and hence
\begin{equation}\label{eq:tpi relation}
    \tpi(s) - \floor[\tpi(s)] = \tpi(t) - \floor[\tpi(t)].
\end{equation}
So to show that $(\varphi(\lambda), \tpi(s)) \sim (\varphi(\mu), \tpi(t))$, it remains to show that
\[
    \varphi(\lambda)(\floor[\tpi(s)], \ceil[\tpi(s)]) = \varphi(\mu)(\floor[\tpi(t)], \ceil[\tpi(t)]).
\]
Since $\floor[s] \le s \le \ceil[s]$, we have $\pi(\floor[s]) \le \tpi(s) \le \pi(\ceil[s])$ and
similarly for $t$. Since $\pi(\floor[s]), \pi(\ceil[s]) \in \NN^l$, it follows from the definition
of the floor and ceiling functions that
\[
\pi(\floor[s]) \le \floor[\tpi(s)] \le \tpi(s) \le \ceil[\tpi(s)] \le \pi(\ceil[s]).
\]
Moreover, since $(\lambda,s) \sim (\mu,t)$ and since $\varphi$ is a $\pi$-quasimorphism, we have
\begin{align*}
\varphi(\lambda)(\pi(\floor[s]), \pi(\ceil[s]))
    &= \varphi(\lambda(\floor[s], \ceil[s])) \\
    &= \varphi(\mu(\floor[t], \ceil[t]))
    = \varphi(\mu)(\pi(\floor[t]), \pi(\ceil[t])).
\end{align*}
Moreover, equation~\eqref{eq:tpi relation} forces $\ceil[\tpi(s)] - \floor[\tpi(s)] =
\ceil[\tpi(t)] - \floor[\tpi(t)]$. Since $(\lambda, s) \sim (\mu,t)$ we have $s - \floor[s] = t -
\floor[t]$ and hence
\[
\floor[\tpi(s)] - \pi(\floor[s]) = \floor[\tpi(t)] - \pi(\floor[t]).
\]
So
\begin{align*}
    \varphi(\lambda)&(\floor[\tpi(s)], \ceil[\tpi(s)]) \\
        &= \big(\varphi(\lambda)(\tpi(\floor[s]), \tpi(\ceil[s]))\big)\big(\floor[\tpi(s)] - \tpi(\floor[s]), \\
            &\hskip12em\floor[\tpi(s)] - \tpi(\floor[s]) + (\ceil[\tpi(s)] - \floor[\tpi(s)])\big) \\
        &= \big(\varphi(\mu)(\tpi(\floor[t]), \tpi(\ceil[t]))\big)\big(\floor[\tpi(t)] - \tpi(\floor[t]), \\
            &\hskip12em\floor[\tpi(t)] - \tpi(\floor[t]) + (\ceil[\tpi(t)] - \floor[\tpi(t)])\big) \\
        &= \varphi(\mu)(\floor[\tpi(t)], \ceil[\tpi(t)]).
\end{align*}
Hence $\tphi$ is well-defined. To see that it is continuous, fix an open subset $U$ of $X_\Gamma$.
By definition of the quotient topology on $X_\Lambda$, to show that $\tphi^{-1}(U)$ is open, we
must show that for each $\lambda \in \Lambda$, the set $\{s \in [0, d(\lambda)] :
\tphi([\lambda,s]) \in U\}$ is open in $[0, d(\lambda)]$. Since $\tphi([\lambda,s]) =
[\varphi(\lambda), \tpi(s)]$, we have
\[
\{s \in [0, d(\lambda)] : \tphi([\lambda,s]) \in U\}
    = \tpi^{-1}\big(\{t \in [0, \pi(d(\lambda))] : [\varphi(\lambda), t] \in U\}\big).
\]
Since $U$ is open in $X_\Gamma$, the set $\{t \in [0, \pi(d(\lambda))] : [\varphi(\lambda), t] \in
U\}$ is open in $[0, \pi(d(\lambda))]$. So continuity of $\tpi$ implies that $\{s \in [0,
d(\lambda)] : \tphi([\lambda,s]) \in U\}$ is open also. That $\varphi' \circ \varphi$ is a $\pi'
\circ \pi$-quasimorphism is routine. For $[\lambda,t] \in X_\Lambda$, we have
\[
\tphi' \circ \tphi([\lambda,t])
    = \tphi'([\varphi(\lambda), \tpi(t)])
    = [\varphi' \circ \varphi(\lambda), \pi' \circ \pi(t)]
    = \big(\varphi \circ \varphi'\big)^{\sim}([\lambda,t]),
\]
which establishes the final assertion and completes the proof.
\end{proof}
\begin{rmk}
Suppose that $k = l$ and $\pi$ is the identity map, so that $\varphi : \Lambda \to \Gamma$ is a
morphism of $k$-graphs. Using the decompositions $X_\Lambda = \bigsqcup_{0 \le d(\lambda)
\le \One{k}} Q_\lambda$ and $X_\Gamma = \bigsqcup_{0 \le d(\gamma) \le \One{l}} Q_\gamma$ of
Lemma~\ref{lem:CW structure}, we see that $\tphi$ is determined by $\tphi([\lambda,t]) =
[\varphi(\lambda), t]$ whenever $d(\lambda) \le \One{k}$ and $t \in (0, d(\lambda))$.
\end{rmk}

\begin{prp}
In addition to the hypotheses of Proposition~\ref{prp:functorial}, suppose that $\pi$ is rectilinear in the
sense that each $\pi(e_i)$ has the form $n_i e_{j_i}$ for some $n_i \in \NN$ and $j_i \le l$, and
suppose also that $\varphi$ is \emph{weakly surjective} in the sense that for each $\gamma \in \Gamma$
there exists $\lambda \in \Lambda$ and $p, q \in \NN^k$ with $p \le q \le \pi(d(\lambda))$ such
that $\gamma = \varphi(\lambda)(p, q)$. Then $\tphi$ is surjective. If $k = l$ and $\pi$ is the
identity so that $\varphi$ is a $k$-graph morphism, and if $\varphi$ is injective, then
$\tphi$ is also injective.
\end{prp}
\begin{proof}
Suppose that $\pi$ is rectilinear and $\varphi$ is weakly surjective. Fix $[\gamma,t] \in
X_\Gamma$. Fix $\lambda \in \Lambda$ and $p \in \NN^l$ such that $\varphi(\lambda)(p, p +
d(\gamma)) = \gamma$. Then $[\gamma, t] = [\varphi(\lambda), p + t]$. By hypothesis on $\pi$ we
have $\pi(d(\lambda)) = \sum^k_{i=1} d(\lambda)_i n_i e_{j_i}$. For $h \le l$ define $ \alpha_h =
(n + t)_h/\pi(d(\lambda))_h$. Let $s = \sum^k_{i=1} \alpha_{j_i} d(\lambda)_i e_i$. Then
\[
\tpi(s) = \sum^k_{i=1} \frac{(n+t)_{j_i}}{\pi(d(\lambda))_{j_i}} d(\lambda)_i \pi(e_i).
\]
Thus, for $h \le l$, we have
\[
\pi(s)_h
    = \sum_{j_i = h} \frac{(n+t)_h}{\pi(d(\lambda))_h} \pi(d(\lambda)_i e_i)
    = \frac{(n+t)_h}{\pi(d(\lambda))_h} \pi(d(\lambda))_h
    = (n+t)_h.
\]
Since each $\alpha_h \in [0,1]$, we have $s \in [0, d(\lambda)]$, and we have $\tphi([\lambda,s]) =
[\varphi(\lambda), n+t] = [\gamma,t]$ as required.

Now suppose that $\varphi$ is an injective $k$-graph morphism. Suppose that $\tphi([\lambda, s]) =
\tphi([\mu,t])$. Then $[\varphi(\lambda), \tpi(s)] = [\varphi(\mu), \tpi(t)]$. In particular,
$\varphi(\lambda)(\floor[s], \ceil[s]) = \varphi(\mu)(\floor[t], \ceil[t])$. Since $\varphi$ is
injective, it follows that $\lambda(\floor[s], \ceil[s]) = \mu(\floor[t], \ceil[t])$. Moreover,
the equality $[\varphi(\lambda), \tpi(s)] = [\varphi(\mu), \tpi(t)]$ implies that $s - \floor[s] = t -
\floor[t]$. Hence $[\lambda, s] = [\mu, t]$.
\end{proof}

In the preceding proposition, the hypothesis used to establish that $\tphi$ is injective could be
weakened to require just that $\pi$ maps generators to generators and $\varphi$ is injective.
However, these two hypotheses imply that $\Lambda$ consists entirely of paths of dimension at most
$l$, and that  $\pi$ is a generator-to-generator injection on degrees of such paths. So locally,
$\varphi$ is just a relabeling of an injective $k$-graph morphism. The next example shows that it
does not suffice to ask merely that each of $\pi$ and $\varphi$ be injective.

\begin{example}
Let $E$ be the $1$-graph
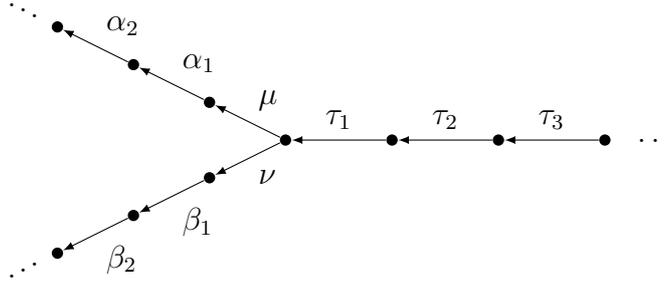
\[
\begin{tikzpicture}[yscale=0.5]
    \node at (-3.3,3.3) {.};
    \node at (-3.45,3.45) {.};
    \node at (-3.6,3.6) {.};
    \node[circle, inner sep=1.5pt, fill=black] (-3+) at (-3,3) {};
    \node[circle, inner sep=1.5pt, fill=black] (-2+) at (-2,2) {};
    \node[circle, inner sep=1.5pt, fill=black] (-1+) at (-1,1) {};
    \node at (-3.3,-3.3) {.};
    \node at (-3.45,-3.45) {.};
    \node at (-3.6,-3.6) {.};
    \node[circle, inner sep=1.5pt, fill=black] (-3-) at (-3,-3) {};
    \node[circle, inner sep=1.5pt, fill=black] (-2-) at (-2,-2) {};
    \node[circle, inner sep=1.5pt, fill=black] (-1-) at (-1,-1) {};
    \node[circle, inner sep=1.5pt, fill=black] (0) at (0,0) {};
    \node[circle, inner sep=1.5pt, fill=black] (1) at (1.4,0) {};
    \node[circle, inner sep=1.5pt, fill=black] (2) at (2.8,0) {};
    \node[circle, inner sep=1.5pt, fill=black] (3) at (4.2,0) {};
    \node at (4.9,0) {\dots};
    \draw[-latex] (3)--(2) node[above, pos=0.5] {$\tau_3$};
    \draw[-latex] (2)--(1) node[above, pos=0.5] {$\tau_2$};
    \draw[-latex] (1)--(0) node[above, pos=0.5] {$\tau_1$};
    \draw[-latex] (0)--(-1+) node[anchor=south west, pos=0.5] {$\mu$};
    \draw[-latex] (-1+)--(-2+) node[anchor=south west, pos=0.5] {$\alpha_1$};
    \draw[-latex] (-2+)--(-3+) node[anchor=south west, pos=0.5] {$\alpha_2$};
    \draw[-latex] (0)--(-1-) node[anchor=north west, pos=0.5] {$\nu$};
    \draw[-latex] (-1-)--(-2-) node[anchor=north west, pos=0.5] {$\beta_1$};
    \draw[-latex] (-2-)--(-3-) node[anchor=north west, pos=0.5] {$\beta_2$};
\end{tikzpicture}
\]
Define $\pi : \NN \to \NN$ by $\pi(n) = 2n$. Define $\varphi : E^* \to E^*$ by
\begin{gather*}
\varphi(\mu) = \mu\tau_1,\quad
    \varphi(\nu) = \nu\tau_1,\quad
    \varphi(\tau_i) = \tau_{2i}\tau_{2i+1},\\
    \varphi(\alpha_i) = \alpha_{2i}\alpha_{2i-1}\quad\text{and}\quad
    \varphi(\beta_i) = \beta_{2i}\beta_{2i-1}.
\end{gather*}
Then $\pi$ is injective, and $\varphi$ is an injective $\pi$-quasimorphism. However, for
$t \in [1/2,1]$, we have $\tphi(\mu,t) = [\tau_1, 2t-1] = \tphi(\nu,t)$. So $\tphi$ is not
injective.
\end{example}

The following result shows that the topological realization functor $(\Lambda\mapsto X_\Lambda,
\varphi\mapsto \wilde\varphi)$ is faithful.

\begin{lem}\label{faithful}
Let $\varphi,\psi:\Lambda\to\Gamma$ be $k$-graphs morphisms such that $\wilde\varphi=\wilde\psi$.
Then $\varphi=\psi$.
\end{lem}

\begin{proof}
The equality $\wilde\varphi=\wilde\psi$ implies that $\varphi$ and $\psi$ agree on commuting
cubes, and in particular on edges. Since $\Lambda$ is generated by its edges, $\varphi$ and $\psi$
must therefore coincide by functoriality (see Proposition~\ref{prp:functorial}).
\end{proof}

\begin{lem}\label{functor group}
Let $\varphi:\Lambda\to\Gamma$ be a morphism of $k$-graphs, and let $\wilde\varphi:X_\Lambda\to
X_\Gamma$ be the associated map between the topological realizations. Let $u\in\Lambda^0$, and let
$v=\varphi(u)$. Then the isomorphisms $\pi_1(\Lambda, u) \cong \pi_1(X_\Lambda, u)$ and
$\pi_1(\Gamma, v) \cong \pi_1(X_\Gamma, v)$ of Corollary~\ref{isomorphism} make the diagram
\[
\xymatrix{
\pi_1(\Lambda,u) \ar[r]^-\cong \ar[d]_{\varphi_*}
&\pi_1(X_\Lambda,u) \ar[d]^{\wilde\varphi_*}
\\
\pi_1(\Gamma,v) \ar[r]_-\cong
&\pi_1(X_\Gamma,v)
}
\]
commute.
\end{lem}
\begin{proof}
Since $\varphi$ is a $k$-graph morphism, it restricts to a morphism $\varphi^1 : E_\Lambda \to
E_\Gamma$ of $1$-skeletons. Proposition~\ref{prp:functorial} implies that $\varphi^1$ induces a
homomorphism $\varphi^1_* : \pi_1(E_\Lambda) \to \pi_1(E_\Gamma)$. Lemma~\ref{1 isomorphism} shows
that $\varphi^1_*$ is compatible with the induced homomorphism $\wilde{\varphi}^1_* :
\pi_1(X_\Lambda^1) \to \pi_1(X_\Gamma^1)$. The result then follows from
Corollary~\ref{isomorphism}.
\end{proof}

\section{Topological realizations and coverings of higher-rank graphs}\label{sec:coverings}

We investigate the relationship between covering maps in the algebraic and topological senses.
We will assume throughout this section that all $k$-graphs are connected and all spaces are connected CW-complexes.

Let $\Lambda$ be a $k$-graph. Recall from \cite{pqr:cover} that a \emph{covering} of $\Lambda$ is a
surjective $k$-graph morphism $p:\Omega\to\Lambda$ such that for all $v\in\Omega^0$, $p$ maps
$\Omega v$ bijectively onto $\Lambda p(v)$ and maps $v\Omega$ bijectively onto $p(v)\Lambda$.

Our main purpose here is to prove the following theorem.

\begin{thm}\label{cover}
If $p:\Omega\to\Lambda$ is a covering of $k$-graphs, then $\wilde p:X_\Omega\to
X_\Lambda$ is a covering map of the topological realizations.
\end{thm}

We know that $\wilde p$ is a continuous surjection.
We must show that $X_\Lambda$ is covered by open sets $U$ that are \emph{evenly covered}, i.e., $\wilde p\inv(U)$ is a disjoint union of open sets that $\wilde p$ maps homeomorphically onto $U$.

\begin{obs}\label{coordinates}
Let $x=[\lambda,t]\in X_\lambda$ with $d(\lambda) \le \One k$ and $t\in Q_\lambda$.
\begin{enumerate}
\item It follows from the covering property of $p$ that for each $y\in p\inv(x)$
    there is a unique $\nu$ with $d(\nu) \le \One{k}$ such that $y\in Q_\nu$ and
    $p(\nu)=\lambda$.

\item For each $i=1,\dots,k$, we have $0<t_i<1$ if $d(\lambda)_i=1$, and $t_i=0$ if
    $d(\lambda)_i=0$.

\item Suppose $d(\mu) \le \One{k}$. Then $x\in\bar Q_\mu$ if and only if there exists $s\le
    d(\mu)$ such that $[\lambda,t]=[\mu,s]$, in which case we have
\begin{enumerate}
\item $s_i=t_i$ if $d(\lambda)_i=1$;
\item $s_i\in \{0,1\}$ if $d(\lambda)_i=0$ and $d(\mu)_i=1$.
\end{enumerate}
\end{enumerate}
\end{obs}

\begin{defn}
Fix $\lambda \in \Lambda$ with $d(\lambda) \le \One{k}$ and $t \in (0, d(\lambda))$. Let $x =
[\lambda, t] \in Q_\lambda$. We define $N_x$ to be the set of all $[\mu,s]\in X_\Lambda$ satisfying
the following conditions:
\begin{enumerate}
\item $d(\mu) \le \One{k}$,
\item $x\in \bar Q_\mu$,
\item $0<s_i<1$ if $d(\lambda)_i=1$, and
\item $|s_i-r_i|<1/2$ if $[\mu,r]=[\lambda,t]$, $d(\lambda)_i=0$, and $d(\mu)_i=1$.
\end{enumerate}
\end{defn}

\begin{lem}\label{open}
$N_x$ is an open neighborhood of $x$.
\end{lem}

\begin{proof}
Taking $\mu = \lambda$ and $s = t$ in the definition of $N_x$ shows that $x\in N_x$. By definition
of the weak topology, it suffices to show that if $d(\mu) \le \One{k}$ then the intersection
$N_x\cap \bar Q_\mu$ is relatively open, and furthermore, in order to show that $N_x\cap \bar
Q_\mu$ is relatively open, it suffices to show that $V_\mu\subset [0,d(\mu)]$ defined by
\[
V_\mu=\{s\in [0,d(\mu)]:[\mu,s]\in N_x\cap \bar Q_\mu\}
\]
is open. We consider three cases.
\begin{enumerate}
\item If $x\not\in \bar Q_\mu$, then $V_\mu=\emptyset$ is open.
\item If $\mu=\lambda$, then
    \[
    V_\mu=\{s\in [0,d(\mu)]:\text{$0<s_i<1$ if $d(\mu)_i=1$}\}.
    \]
\item If $\mu\ne\lambda$ and $x\in \bar Q_\mu$, then $|\mu|>|\lambda|$, and $x$ is in the
    boundary of the open cell $Q_\mu$. For each $i\in\{1,\dots,k\}$ write
    \[
    V_\mu^i=\{s_i:s\in V_\mu\}.
    \]
    Then $V_\mu=\prod_{i=1}^kV_\mu^i$. So it suffices to show that each $V_\mu^i$ is
    relatively open in $[0, d(\mu)_i]$. Fix $i \le k$. If $d(\mu)_i=0$ then
    $V_\mu^i= \{0\} = [0, d(\mu)_i]$. If $d(\lambda)_i=1$ then $V_\mu^i=(0,1)$. So
    we turn to the remaining case $d(\lambda)_i=0$ and $d(\mu)_i=1$. Then $V_\mu^i=[0,1/2)$ or
    $(1/2,1]$ \emph{except} in the following two circumstances:
    \begin{enumerate}
    \item $|\lambda|=0$ and there exists $n\in\N^k$ such that
    \[
    \lambda=\mu(n)=\mu(n+e_i),
    \]
    or

    \item $|\lambda|=1$ and there exists $n\in\N^k$ such that
    \[
    \lambda=\mu(n,n+e_i)
    \midtext{and}
    \mu(n)=\mu(n+e_i);
    \]
    \end{enumerate}
    in each of (a) and (b) we have $V_\mu^i=[0,1/2)\cup (1/2,1]$. We have shown that in every
    case $V_\mu^i$ is an open subset of $[0,d(\mu)_i]$.\qedhere
\end{enumerate}
\end{proof}

\begin{lem}\label{even}
$N_x$ is evenly covered.
\end{lem}

\begin{proof}
Since the map $\wilde p:X_\Omega\to X_\Lambda$ has the form $\wilde p([\mu,t])=[p(\mu),t]$, the
inverse image $\wilde p\inv(N_x)$ is the disjoint union over $y\in \wilde p\inv(x)$ of the
corresponding neighborhoods $N_y$.

We must show that:
\begin{enumerate}
\item for each $y\in \wilde p\inv(x)$, the map $\wilde p$ restricts to a
    homeomorphism of $N_y$ onto $N_x$; and
\item for distinct $y,z\in \wilde p\inv(x)$ we have $N_y\cap N_z=\emptyset$.
\end{enumerate}

For (1), let $q_\Lambda:\bigsqcup_{d(\mu) \le \One{k}}\{k\}\times [0,d(\mu)]\to X_\Lambda$ be the
quotient map, and similarly for $q_\Omega$. We have
\[
q_\Lambda\inv(N_x)=\bigsqcup_{d(\mu) \le \One{k}}\{k\}\times V_\mu,
\]
where $V_\mu$ is open in $[0,d(\mu)]$ for each $\mu$, and similarly
\[
q_\Omega\inv(N_y)=\bigsqcup_{d(\nu) \le\One{k}}\{k\}\times V_\nu^y,
\]
where $V_\nu^y$ is open in $[0,d(\nu)]$ for each $\nu$.
%Note that $V_\nu^y=V_{p(\nu)}$.

Define $p':\bigsqcup_{d(\nu) \le \One{k}}\{k\}\times V_\nu^y\to \bigsqcup_{d(\mu) \le
\One{k}}\{k\}\times V_\mu$ by $p'(\mu,t)=(p(\mu),t)$. Then $p'$ is a homeomorphism, because
\[
p'\bigl(\{\nu\}\times V_\nu^y\bigr)=\{p(\nu)\}\times V_{p(\nu)}.
\]
Also, $(\nu,t)\sim (\omega,r)$ in $\bigsqcup_{d(\nu) \le \One{k}}\{k\}\times V_\nu^y$ if and only
if $p'(\nu,t)\sim p'(\omega,r)$ in $\bigsqcup_{d(\nu) \le \One{k}}\{k\}\times V_{p(\nu)}^y$.

For (2), suppose not, and take $w\in N_y\cap N_z$.
Let $\mu$ and $\nu$ be the unique elements of $p\inv(\lambda)$ such that $y\in Q_\mu$ and $z\in Q_\nu$.
Then $y=[\mu,t]$ and $z=[\nu,t]$.
Let $\alpha$ be the unique cube in $\Omega$ such that $w\in Q_\alpha$,
and let $s$ be the unique element of $(0,d(\alpha))$ such that $w=[\alpha,s]$.
By \obsref{coordinates} (1) we cannot have $p(\alpha)=\lambda$.
Therefore we must have $|\alpha|>|\mu|$.
Since $y\in \bar Q_\alpha$, by \obsref{coordinates} (1) there exists $a\le d(\alpha)$ such that
$y=[\alpha,a]$, and $a_i=t_i$ except for those $i$ for which $d(\mu)_i=0$ and $d(\alpha)_i=1$.
Similarly, there exists $b\le d(\alpha)$ such that $z=[\alpha,b]$, and $b_i=t_i$ except when $d(\nu)_i=0$ and $d(\alpha)_i=1$.
Since $y\ne z$, there exists $i$ such that $a_i\ne b_i$,
and then we must have $d(\mu)_i=d(\nu)_i=0$ and $d(\alpha)_i=1$.
Since $w\in N_y\cap N_z$, we have $|s_i-a_i|<1/2$ and $|s_i-b_i|<1/2$.
But this is a contradiction since $0<s_i<1$ and $a_i$ and $b_i$ are distinct integers.
\end{proof}

\begin{proof}[Proof of \thmref{cover}]
This follows from \lemref{open} and \lemref{even}, because $\Lambda$ is covered by the open sets $N_x$ for $x\in X_\Lambda$.
\end{proof}

\begin{lem}\label{cover is realization}
Let $\Lambda$ be a $k$-graph, and let $q:Y\to X_\Lambda$ be a covering map.
Then there are a $k$-graph $\Omega$, a covering $p:\Omega\to\Lambda$ and a homeomorphism
$\phi: Y \to X_\Omega$ such that $q = \wilde p \circ \phi$.
%%the covering map $\wilde p:X_\Omega\to X_\Lambda$ is isomorphic to $q$.
\end{lem}

\begin{proof}
Choose $u\in\Lambda^0$ and $v\in q\inv(u)$.
Let $H'$ be the subgroup $q_*(\pi_1(Y,v))$ of $\pi_1(X_\Lambda,u)$,
and let $H$ be the subgroup of $\pi_1(\Lambda,u)$ corresponding to $H'$ under the isomorphism of \corref{isomorphism}.
By \cite[Theorem~2.8]{pqr:cover}, there are a connected $k$-graph $\Omega$, a covering $p:\Omega\to\Lambda$, and $w\in p\inv(u)$ such that $H=p_*(\pi_1(\Omega,w))$.
Then $\wilde p: X_\Omega\to X_\Lambda$ is a covering map, and by \lemref{functor group} we have
a commuting diagram
\[
\xymatrix{
\pi_1(\Omega,w) \ar[r]^-{\cong} \ar[d]_{p_*}
&\pi_1(X_\Omega,w) \ar[d]^{\wilde p_*}
\\
\pi_1(\Lambda,u) \ar[r]_-{\cong}
&\pi_1(X_\Lambda,u).
}
\]
Thus $\wilde p_*(\pi_1(X_\Omega,w))=H'$, so by \cite[Corollary~V.6.4]{massey}, the coverings $(Y,
q)$ and $(X_\Omega, \wilde p)$ are isomorphic, that is, there is a homeomorphism $\phi: Y \to
X_\Omega$ such that $q = \wilde p \circ \phi$.
\end{proof}

For a fixed $k$-graph $\Lambda$, we have a category $\algcov\Lambda$ of coverings of
$\Lambda$, and we also have a category $\topcov\Lambda$ of coverings of the topological
realization. Each morphism
\[
\xymatrix{
\Omega \ar[rr]^-\varphi \ar[dr]_p
&&\Gamma \ar[dl]^q
\\
&\Lambda
}
\]
in $\algcov\Lambda$ determines a morphism (also called a deck transformation)
\[
\xymatrix{
X_\Omega \ar[rr]^-{\wilde\varphi} \ar[dr]_{\wilde p}
&&X_\Gamma \ar[dl]^{\wilde q}
\\
&X_\Lambda
}
\]
in $\topcov\Lambda$.

\begin{thm}\label{equivalence}
With the above notation, the assignments $(\Omega,p)\mapsto (X_\Omega,\wilde p)$ and
$\varphi\mapsto\wilde\varphi$ give an equivalence
$\Phi:\algcov\Lambda\xrightarrow{\sim}\topcov\Lambda$.
In particular, if $(\Omega,p)$ is universal cover of $\Lambda$, then $(X_\Omega,\wilde p)$
is a universal cover of $X_\Lambda$.
\end{thm}

\begin{proof}
$\Phi$ is functorial because $\Lambda\mapsto X_\Lambda$ is.
We must show that $\Phi$ is
\begin{enumerate}
\item faithful,
\item full, and
\item essentially surjective.
\end{enumerate}

(1) follows from \corref{faithful}.

For (2), let $p:\Omega\to\Lambda$ and $q:\Gamma\to\Lambda$ be coverings, and suppose that
$\psi:(X_\Omega,\wilde p)\to (X_\Gamma,\wilde q)$ is a morphism. Choose $v\in\Omega^0$, and let
$u=p(v)$ and $w=\psi(v)$. We have $\wilde q(w) = \wilde q \circ \psi(v) = \wilde p(v) \in
\Lambda^0$. Since $q$ preserves degree, $\wilde q$ maps open $n$-cubes to open $n$-cubes, and in
particular $\wilde{q}^{-1}(\Lambda^0) = \Gamma^0$. So $w \in \Gamma^0$. We have
\[
\wilde q_*\circ \psi_*=\wilde p_*:\pi_1(X_\Omega,v)\to \pi_1(X_\Lambda,u)),
\]
so
\[
\wilde p_*(\pi_1(X_\Omega,v))\subset \wilde q_*(\pi_1(X_\Gamma,w)),
\]
and hence
\[
p_*(\pi_1(\Omega,v))\subset q_*(\pi_1(\Gamma,w)).
\]
Thus by \cite[Theorem~2.2]{pqr:cover} there is a unique morphism $\varphi:(\Omega,p)\to (\Gamma,q)$
taking $v$ to $w$. Then both $\wilde \varphi$ and $\psi$ are morphisms from $(X_\Omega,\wilde p)$
to $(X_\Gamma,\wilde q)$ taking $v$ to $w$, and hence must coincide, by
\cite[Lemma~6.3]{massey}.\footnote{The two quoted references don't explicitly address uniqueness of
morphisms, but this follows by uniqueness of liftings.}

For (3), we must show that every object $(Y,q)$ in $\topcov\Lambda$ is isomorphic to one in the
image of $\Phi$. But this is exactly what \lemref{cover is realization} says.

The final assertion follows from the universal properties of universal covers. By \cite[Theorem
2.7]{pqr:cover} $\Lambda$ has a universal cover $(\Omega, p)$. Let $u \in \Omega^0$. Then by
\cite[Theorem 2.7]{pqr:cover} $p_*\pi_1(\Omega, u)$ is trivial. Hence, \lemref{functor group}
implies that $\wilde p_*\pi_1(X_\Omega, u)$ is also trivial.  It follows that $X_\Omega$ is simply
connected and therefore $X_\Omega$ is a universal cover of $X_\Lambda$.
\end{proof}

\begin{rmk}
Let $(\Omega,p)$ be a covering and fix $v \in \Omega^0$.
Generalizing from the context of directed graphs (see \cite{DPR}), we call
$(\Omega,p)$ \emph{regular} if
$p_*(\pi_1(\Omega,v))$ is a normal subgroup of $\pi_1(\Lambda,p(v))$.  See \cite[Corollary 2.4]{pqr:cover}
for a number of equivalent conditions.
The corresponding property of topological coverings is well-known (see \cite[p.~134]{massey}, for example).
Using \lemref{functor group}, it is easy to verify that the covering $(\Omega,p)$ is regular if and only if the topological realization $(X_\Omega,\wilde p)$ is.
\end{rmk}

\begin{example}[{\cite[Example~5.8]{KPS3}}]\label{ex:big sphere}
Arguing as in Examples \ref{ex:sphere}--\ref{ex:Klein}, we see that the topological realization of the $2$-graph with the  skeleton
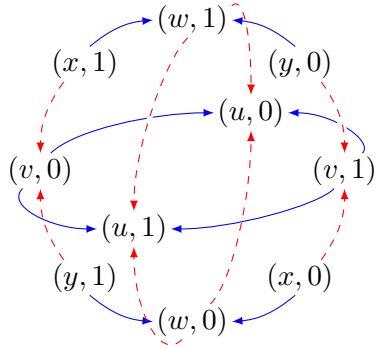
\[
\begin{tikzpicture}[scale=2]
    \node[inner sep= 1pt] (w1) at (xyz spherical cs:longitude=0,latitude=0,radius=1) {\small$(w,1)$};
    \node[inner sep= 1pt] (y0) at (xyz spherical cs:longitude=45,latitude=0,radius=1) {\small$(y,0)$};
    \node[inner sep= 1pt] (v1) at (xyz spherical cs:longitude=90,latitude=0,radius=1) {\small$(v,1)$};
    \node[inner sep= 1pt] (x0) at (xyz spherical cs:longitude=135,latitude=0,radius=1) {\small$(x,0)$};
    \node[inner sep= 1pt] (w0) at (xyz spherical cs:longitude=180,latitude=0,radius=1) {\small$(w,0)$};
    \node[inner sep= 1pt] (y1) at (xyz spherical cs:longitude=225,latitude=0,radius=1) {\small$(y,1)$};
    \node[inner sep= 1pt] (v0) at (xyz spherical cs:longitude=270,latitude=0,radius=1) {\small$(v,0)$};
    \node[inner sep= 1pt] (x1) at (xyz spherical cs:longitude=315,latitude=0,radius=1) {\small$(x,1)$};
    \node[inner sep= 1pt] (u1) at (xyz spherical cs:longitude=0,latitude=90,radius=1) {\small$(u,1)$};
    \node[inner sep= 1pt] (u0) at (xyz spherical cs:longitude=0,latitude=-90,radius=1) {\small$(u,0)$};
    \draw[-latex, blue, out=40, in=185] (x1.75) to (w1.west);
    \draw[-latex, red, dashed, out=235, in=90] (x1.220) to (v0.north);
    \draw[-latex, red, dashed, out=125, in=270] (y1.140) to (v0.south);
    \draw[-latex, blue, out=320, in=175] (y1.285) to (w0.west);
    \draw[-latex, blue, out=220, in=5] (x0.255) to (w0.east);
    \draw[-latex, red, dashed, out=55, in=270] (x0.40) to (v1.south);
    \draw[-latex, blue, out=140, in=355] (y0.105) to (w1.east);
    \draw[-latex, blue] (v0.60) .. controls +(0,0,-0.4) and +(-0.4,0,0) .. (u0.west);
    \draw[-latex, blue] (v0.225) .. controls +(0,0,0.3) and +(-0.3,0,0) .. (u1.west);
    \draw[-latex, blue] (v1.240) .. controls +(0,0,0.3) and +(0.4,0,0) .. (u1.east);
    \draw[-latex, blue] (v1.45) .. controls +(0,0,-0.2) and +(0.3,0,0) .. (u0.east);
    \draw[white, line width=2pt] (w1.255) .. controls +(0,0,0.4) and +(0,0.4,0) .. (u1.north);
    \draw[-latex, red, dashed] (w1.255) .. controls +(0,0,0.4) and +(0,0.4,0) .. (u1.north);
    \draw[-latex, red, dashed] (w0.240) .. controls +(0,0,0.4) and +(0,-0.4,0) .. (u1.south);
    \draw[-latex, red, dashed] (w0.45) .. controls +(0,0,-0.4) and +(0,-0.4,0) .. (u0.south);
    \draw[-latex, red, dashed] (w1.20) .. controls +(0,0,-0.2) and +(0,0.4,0) .. (u0.north);
% These things drawn last for masking
    \draw[white, line width=3pt, out=305, in=90] (y0.320) to (v1.north);
    \draw[-latex, red, dashed, out=305, in=90] (y0.320) to (v1.north);
\end{tikzpicture}
\]
is a sphere.
Moreover the action $\alpha$ of $\ZZ/2\ZZ$ on $\Lambda$ that interchanges opposite vertices
induces the antipodal map on the topological realization, so the quotient is the projective plane.
Indeed, as discussed in \cite{KPS3}, this 2-graph is a skew product of the $2$-graph of
Example~\ref{eg:projective} by $\ZZ/2\ZZ$, the action $\alpha$ is then translation in $\ZZ/2\ZZ$ in
the skew product, and hence the quotient $2$-graph is exactly the $2$-graph of
Example~\ref{eg:projective}.
\end{example}

\begin{rmk}
Interestingly, although the 2-graphs in Examples \ref{ex:sphere}~and~\ref{ex:big sphere} have
topological realizations homeomorphic to the surface of a 3-cube, it is not hard to check that
there is no 2-graph whose cell complex consists of the faces of the cube.
\end{rmk}

\section{Towers of coverings and projective limits}\label{sec:towers}

As in \cite{KPS1}, fix a sequence $(\Lambda_n)^\infty_{n=0}$ of row-finite $k$-graphs with no
sources and a sequence $(p_n)^\infty_{n=1}$ of finite-to-one coverings $p_n : \Lambda_n \to
\Lambda_{n-1}$. There is a unique $(k+1)$-graph $\Sigma = \Sigma(\Lambda_n, p_n)$\footnote{In
\cite{KPS1} this $(k+1)$-graph was denoted $\tgrphlim(\Lambda_n, p_n)$. But in this paper we shall
be discussing the construction of \cite{KPS1} in close proximity with projective limits of
topological spaces, so the notation of \cite{KPS1} would be very confusing here. The notation
$\Sigma(\Lambda_n, p_n)$ is reminiscent of the notation for the linking graph associated to the
$\Omega_1$-system of $k$-morphs obtained from the system $(\Lambda_n, p_n)$ (see
\cite[Examples~5.3(iv)]{KPS2}).} such that $\Sigma^0 = \bigsqcup^\infty_{n=0} \Lambda^0_n$,
$\Sigma^{e_i} = \bigsqcup_{n=0}^\infty \Lambda^{e_i}_n$ for $i \le k$, and $\Sigma^{e_{k+1}} =
\bigsqcup^\infty_{n=1} \{f_v : v \in \Lambda^0_n\}$, and with structure maps on
$\bigsqcup^\infty_{n=0} \Lambda_n \subseteq \Sigma$ inherited from the $\Lambda_n$, range and
source on $\Sigma^{e_{k+1}}$ given by $s(f_v) = v$ and $r(f_v) = p_n(v)$ for $v \in \Lambda^0_n$,
and factorization rules for edges of degree $e_{k+1}$ determined by $f_{r(\lambda)}\lambda =
p(\lambda)f_{s(\lambda)}$ (the unique path-lifting property ensures that this specifies a valid
factorization property). See \cite[Proposition~2.7 and Corollary~2.15]{KPS1} for details.

For $0 \le m \le n$, we write $p^n_m$ for the map $p_{m+1} \circ \cdots \circ p_n : \Lambda_n \to
\Lambda_m$. For each $n \ge 0$ and each $v \in \Lambda^0_n \subseteq \Sigma^0$, the path $F_v :=
f_{p^n_0(v)} f_{p^n_1(v)} \dots f_{p_n(v)}$ is the unique path $F_v \in \Lambda_0^0 \Sigma v$ such
that $d(F_v) \in \NN e_{k+1}$.

It is also shown in \cite{PQS} that given a system $(\Lambda_n, p_n)$ as above, the projective
limit
\[\textstyle
    \varprojlim (\Lambda_n, p_n) = \{(\lambda_n)^\infty_{n=0} \in \prod^\infty_{n=0} \Lambda_n : p_n(\lambda_n) = \lambda_{n-1}\text{ for all } n \ge 1\}
\]
of the discrete spaces $\Lambda_n$ forms a topological $k$-graph in the sense of Yeend
\cite{YeendTopGraph} with structure maps defined coordinatewise and degree map given by
$d\big((\lambda_n)^\infty_{n=1}\big) = d(\lambda_0)$. The thrust in \cite{PQS} is that Yeend's
topological-graph $C^*$-algebra of $\varprojlim (\Lambda_n, p_n)$ is isomorphic to a full corner in
the $k$-graph algebra $C^*(\Sigma)$. Here we are interested in topological aspects of the two
constructions.

We shall show that the fundamental group of $\Sigma$ is identical to that of $\Lambda_0$. We will
also propose a natural notion of the topological realization $X_\Lambda$ of a topological $k$-graph
$\Lambda$ and then show that $X_{\varprojlim (\Lambda_n, p_n)}$ is homeomorphic to the projective
limit $\varprojlim (X_{\Lambda_n}, \wilde{p}_n)$ of the topological realizations of the $\Lambda_n$
under the induced coverings arising from functoriality of the topological-realization construction.
We regard this as evidence that our proposed notion of the topological realization of a topological
$k$-graph is a reasonable one in the sense that it ensures that topological realization is
continuous with respect to projective limits. We deduce that the fundamental group of
$X_{\varprojlim (\Lambda_n, p_n)}$ is isomorphic to the projective limit of the fundamental groups
of the $X_{\Lambda_n}$.

\begin{lem}
Let $(\Lambda_n, p_n)$ be a system of coverings of $k$-graphs and let $\Sigma =
\Sigma(\Lambda_n, p_n)$ as above. Suppose that $w \in \GG(\Sigma)$ satisfies $r(w) \in \Lambda_0$.
Then $w = w' F_{s(w)}$ for some $w' \in \GG(\Lambda_0)$. Moreover for any $v \in \Lambda^0_0$ we
have $v \GG(\Sigma) v = v\GG(\Lambda_0)v$.
\end{lem}
\begin{proof}
Fix $w \in \GG(\Sigma)$ with $r(w) \in \Lambda^0$. Write $w = \lambda_0 \lambda^{-1}_1 \lambda_2
\dots \lambda^{(-1)^n}_n$ with each $\lambda_i \in \Sigma$ (we can always do this, by setting
$\lambda_0 = r(w)$ if necessary). We argue by induction on $n$.

For the base case $n = 0$, consider $\lambda_0 \in \Sigma$ with $r(\lambda) \in \Lambda_0$. Let $p
= d(\lambda)_{k+1}$ and let $m = d(\lambda) - p e_{k+1}$. By the factorization property, $\lambda =
\mu\nu$ for some $\mu \in \Sigma^m$ and $\nu \in \Sigma^{p e_{k+1}}$. Since $d(\mu)_{k+1} = 0$, we
have $\mu \in \bigsqcup_{n=0}^\infty \Lambda_n$, and since $r(\mu) \in \Lambda^0_0$, we then have
$\mu \in \Lambda^0_0$. In particular, $s(\mu) \in \Lambda_0^0$, and hence $r(\nu) \in \Lambda_0^0$.
Moreover $s(\nu) = s(w)$, so $\nu \in \Lambda_0^0 \Sigma s(w)$ with $d(\nu) \in \NN e_{k+1}$. Since
$F_{s(w)}$ is the unique such path, setting $w' = \mu \in \GG(\Lambda_0)$, we have $w = w'
F_{s(w)}$ as required.

Now fix $n \ge 1$ and suppose that $w$ can be written in the desired form whenever $w = \lambda_0
\lambda^{-1}_1 \lambda_2 \dots \lambda^{(-1)^{n-1}}_{n-1}$ for some $\lambda_i \in \Sigma$. Fix an
element $\lambda_0 \lambda^{-1}_1 \lambda_2 \dots \lambda^{(-1)^n}_n$. Let $v =
s(\lambda^{(-1)^{n-1}}_{n-1})$. Applying the inductive hypothesis to $\lambda_0 \lambda^{-1}_1
\lambda_2 \dots \lambda^{(-1)^{n-1}}_{n-1}$ we obtain $w = z F_{v} \lambda^{(-1)^n}_n$ for some $z
\in \GG(\Lambda_0)$. We now consider two cases: $(-1)^n = 1$ or $(-1)^n = -1$.

First suppose that $(-1)^n = 1$. Then $\lambda^{(-1)^n}_n = \lambda_n$ with $r(\lambda_n) = v$, and
we have $w = z F_v \lambda_n$. By the factorization property, we can express $F_v \lambda_n =
\mu\eta$ where $d(\eta) = d(F_v) + d(\lambda_n)_{k+1} e_{k+1}$. We then have $d(\mu)_{k+1} = 0$,
and since $r(\mu) = s(w') \in \Lambda^0_0$ we then have $s(\mu) \in \Lambda^0_0$, and it follows as
in the base case that $\eta = F_{s(w)}$. Hence $w = (z\mu) F_{s(w)}$ has the desired form.

Now suppose that $(-1)^n = -1$, so $\lambda^{(-1)^n}_n = \lambda^{-1}_n$, with $s(\lambda_n) = v$.
Factorize $\lambda_n = \nu\mu$ where $d(\nu) = d(\lambda_n)_{k+1} e_{k+1}$; so $w = z F_v \mu^{-1}
\nu^{-1}$. Let $q$ be the integer such that $v \in \Lambda^0_q$. By definition of the factorization
rules in $\Sigma$, we have $F_{r(\mu)} \mu = p^q_0(\mu) F_{s(\mu)} = p^q_0(\mu) F_v$. Let $\mu_0 =
p^q_0(\mu)$, and let $\gamma = F_{r(\mu)}\mu = \mu_0 F_v$. Then
\[
    F_v \mu^{-1} = F_v \gamma^{-1} \gamma \mu^{-1} = \mu_0^{-1} F_{r(\mu)}.
\]
Hence $w = z \mu_0^{-1} F_{r(\mu)} \nu^{-1}$. Then $w' := z \mu_0^{-1}$ belongs to
$\GG(\Lambda_0)$. Since $d(\nu) = |d(\nu)|e_{k+1}$ and $s(\nu) = s(F_{r(\mu)}) = r(\mu)$, if we
write $m$ for the integer such that $r(\mu) \in \Lambda^0_m$, then $\nu = f_{p^m_q(r(\mu))}\dots
f_{p_m(r(\mu))}f_{r(\mu)}$. In particular,
\[
    F_{r(\mu)} = F_{p^m_q(r(\mu))} \nu = F_{r(\nu))} \nu = F_{s(w)} \nu.
\]
Thus $w = z \mu_0^{-1} F_{s(w)} \nu \nu^{-1} = (z\mu_0) F_{s(w)}$ has the required form. The first
assertion of the lemma now follows by induction. For the second statement, observe that if $v \in
\Lambda_0^0$, then $F_v = v$.
\end{proof}

Recall that a topological $k$-graph is a small category equipped with a second-countable locally
compact Hausdorff topology and a continuous degree map $d : \Lambda \to \NN^k$ satisfying the
factorization property such that $r : \Lambda \to \Lambda^0$ is continuous, $s : \Lambda \to
\Lambda^0$ is a local homeomorphism, and composition is continuous on the space of composable pairs
in $\Lambda$ regarded as a subspace of $\Lambda \times \Lambda$.

\begin{defn}
Let $\Lambda$ be a topological $k$-graph. Let $Y_\Lambda = \{(\lambda, n) \in \Lambda \times \RR^k
: 0 \le n \le d(\lambda)\}$, and endow $Y_\Lambda$ with the relative topology induced by the
product topology on $\Lambda \times \RR^k$. The formula~\eqref{eq:equiv rel} determines an
equivalence relation on $Y_\Lambda$ just as in Section~\ref{sec:topological realization}, and we
define $X_\Lambda = Y_\Lambda/\sim$ endowed with the quotient topology. We call $X_\Lambda$ the
\emph{topological realization of $\Lambda$}.
\end{defn}

Now recall from \cite[Section~6]{PQS} that if $(\Lambda_n, p_n)$ is a system of coverings, then the
topological projective limit
\[\textstyle
    \varprojlim (\Lambda_n, p_n) = \{(\lambda_n)^\infty_{n=1} \in \prod^\infty_{n=1} \Lambda_n : p_n(\lambda_n) = \lambda_{n-1}\text{ for all }n\}
\]
is a topological $k$-graph when endowed with pointwise operations.

\begin{prop}
Let $(\Lambda_n, p_n)$ be  a system of coverings of $k$-graphs. Then there is a homeomorphism
\[
\wilde\pi_\infty : X_{\varprojlim (\Lambda_n, p_n)} \to \varprojlim (X_{\Lambda_n}, (p_n)_*)
\]
such that $\wilde\pi_\infty([(\lambda_i)^\infty_{i=0}, t]) = \big([\lambda_i,
t]\big)^\infty_{i=0}$.
\end{prop}
\begin{proof}
We first construct continuous surjections $\wilde{\pi}_n : X_{\varprojlim (\Lambda_n, p_n)} \to
X_{\Lambda_n}$ such that $(p_n)_* \circ \wilde{\pi}_n = \wilde{\pi}_{n-1}$ for all $n \ge 1$. Fix
$n \in \NN$. Define $\pi^0_n : \bigcup_{m \in \NN^k} (\varprojlim (\Lambda^m_n, p_n)) \times [0,m]
\to Y_{\Lambda_n}$ by $\pi^0_n((\lambda_i)^\infty_{i=1}, t) = (\lambda_n, t)$. A basic open set in
$\Lambda^m_n \times [0,m]$ has the form $U \times B(t;\varepsilon)$ where $U \subseteq \Lambda^m_n$
for some $m \in \NN^k$ is open, $t \in [0,m]$, $\varepsilon > 0$, and the ball
$B(t;\varepsilon)$ is calculated in the metric space $[0,m]$. We have $(\pi^0_n)^{-1}(U \times
B(t;\varepsilon)) = Z(U,n) \times B(t; \varepsilon)$, where $Z(U,n)$ is the cylinder set
$\{(\lambda_i)^\infty_{i=1} \in \varprojlim (\Lambda_n, p_n))^m : \lambda_n \in U\}$. Since this
preimage is open, $\pi^0_n$ is continuous. Now define $\pi_n : (\varprojlim (\Lambda_n, p_n))^m
\times [0,m] \to X_{\Lambda_n}$ by $\pi_n = q \circ \pi^0_n$ where $q : Y_{\Lambda_n} \to
X_{\Lambda_n}$ is the quotient map. Then $\pi_n$ is also continuous. We claim that the formula
\[
    \wilde\pi_n([(\lambda_i)^\infty_{i=1}, t]) = \pi_n((\lambda_i)^\infty_{i=1}, t)
\]
determines a well-defined map $\wilde\pi_n : X_{\varprojlim (\Lambda_n, p_n)} \to X_{\Lambda_n}$.
Indeed, suppose that $[(\lambda_i)^\infty_{i=1}, t] = [(\mu_i)^\infty_{i=1}, s]$. Then $t - \lfloor
t \rfloor = s - \lfloor s \rfloor$, and
\[
\big(\lambda_i(\lfloor t \rfloor, \lceil t \rceil)\big)^\infty_{i=1}
    = (\lambda_i)^\infty_{i=1}(\lfloor t \rfloor, \lceil t \rceil)
    = (\mu_i)^\infty_{i=1}(\lfloor s \rfloor, \lceil s \rceil)
    = \big(\mu_i(\lfloor s \rfloor, \lceil s \rceil)\big)^\infty_{i=1}.
\]
In particular, $\lambda_n(\lfloor t \rfloor, \lceil t \rceil) = \mu_i(\lfloor s \rfloor, \lceil s
\rceil)$. Hence
\[
    \pi_n((\lambda_i)^\infty_{i=1}, t) = [\lambda_n, t] = [\mu_n,s] = \pi_n((\mu_i)^\infty_{i=1}, s),
\]
so $\wilde\pi_n$ is well-defined as claimed. Since $\pi_n$ is continuous, the definition of the
quotient topology on $X_{\varprojlim (\Lambda_n, p_n)}$ ensures that $\wilde\pi_n$ is continuous
too. Since the canonical map $P_n : \varprojlim (\Lambda_n, p_n) \to \Lambda_n$ is surjective for
each $n$, the map $\wilde\pi_n$ is also surjective. By definition of $(p_n)_*$, we have
\[
(p_n)_* \circ \wilde\pi_n([(\lambda_i)^\infty_{i=1}, t])
    = (p_n)_*([\lambda_n, t])
    = [\lambda_{n-1}, t]
    = \wilde\pi_{n-1}([(\lambda_i)^\infty_{i=1}, t]).
\]
The universal property of the projective limit $\varprojlim (X_{\Lambda_n}, (p_n)_*)$ now gives a
unique continuous surjection $\wilde\pi_\infty : X_{\varprojlim (\Lambda_n, p_n)} \to \varprojlim
(X_{\Lambda_n}, (p_n)_*)$ defined by
\[
    \wilde\pi_\infty([(\lambda_i)^\infty_{i=0}, t])
        = (\wilde\pi_n[(\lambda_i)^\infty_{i=0}, t])^\infty_{n=1}
        = [\lambda_n, t]^\infty_{n=0}.
\]

To complete the proof we must show that $\wilde\pi_\infty$ is injective with continuous inverse.
For this fix $([\lambda_n, t_n])^\infty_{n=0} \in \varprojlim (X_{\Lambda_n}, (p_n)_*)$. Then
$[p_n(\lambda_n), t_n] = [\lambda_{n-1}, t_{n-1}]$ for all $n$. For $n \ge 0$, let $\mu_n =
\lambda_n(\lfloor t_n \rfloor, \lceil t_n \rceil)$ and $s_n = t_n - \lfloor t_n \rfloor$. Fix $n
\ge 1$. By definition of the equivalence relation defining the $X_{\Lambda_n}$ we have $s_n =
s_{n-1}$ and
\[
p_n(\mu_n)
    = p_n(\lambda_n)(\lfloor t_n \rfloor, \lceil t_n \rceil)
    = \lambda_{n-1}(\lfloor t_{n-1} \rfloor, \lceil t_{n-1} \rceil)
    = \mu_{n-1}.
\]
Let $s = s_0$. Then $s_n = s_0$ for all $n$, and $([\lambda_n, t_n])^\infty_{n=0} = ([\mu_n,
s])^\infty_{n=0} = \wilde\pi_\infty([(\mu_n)^\infty_{n=0}, s])$. So we may define $\theta :
\varprojlim (X_{\Lambda_n}, (p_n)_*) \to X_{\varprojlim (\Lambda_n, p_n)}$ by $([\lambda_n,
t_n])^\infty_{n=0} \mapsto [(\mu_n)^\infty_{n=0}, s]$ where the $\mu_n$ and $s$ are obtained from
the $\lambda_n$ and $t_n$ as above. The above argument establishes that $\wilde\pi_\infty \circ
\theta$ is the identity map on $X_{\varprojlim (\Lambda_n, p_n)}$. Hence $\wilde{\pi}_\infty$ is
surjective. On the other hand
\begin{align*}
\theta \circ \wilde\pi_\infty\big([(\lambda_n)^\infty_{n=0}, t]\big)
    &= \big[(\lambda_n(\lfloor t \rfloor, \lceil t \rceil))^\infty_{n=0}, t - \lfloor t \rfloor\big] \\
    &= [(\lambda_n)^\infty_{n=0}, t].
\end{align*}
Hence $\theta$ is an algebraic inverse for $\wilde\pi_\infty$. To see that $\theta$ is continuous,
it is enough, as for the other direction, to observe that if $\lambda \in \Lambda_n$ and $U$ is
open in $[0,d(\lambda)]$, then
\begin{align*}
\theta^{-1}(\{[(\mu_i)^\infty_{i=0}, t] \in X_{\varprojlim (\Lambda_i, p_i)} :&{} \mu_n = \lambda, t \in U\}) \\
    &= \wilde\pi_\infty(\{[(\mu_i)^\infty_{i=0}, t] \in X_{\varprojlim (\Lambda_i, p_i)} : \mu_n = \lambda, t \in U\}) \\
    &= \{([\mu_i, t])^\infty_{i=0} : \mu_n = \lambda, t \in U\} \\
    &= Z(\{[\lambda, t] : t \in U\}, n).
\end{align*}
So the preimage under $\theta$ of a sub-basic open set in the image of any connected component of
$Y_{\varprojlim(\Lambda_n, p_n)}$ is the cylinder set of the image of a basic open set in some
component of some $Y_{\Lambda_n}$. Continuity of $\theta$ then follows from the definition of the
quotient topology.
\end{proof}

\begin{cor}
Let $(\Lambda_n, p_n)$ be a system of coverings of $k$-graphs. Then $\pi_1(X_{\varprojlim
(\Lambda_n, p_n)}) \cong \varprojlim(\pi_1(X_{\Lambda_n}), \wilde{(p_n)}_*) \cong
\varprojlim(\pi_1(\Lambda_n), (p_n)_*)$.
\end{cor}
\begin{proof}
By \cite[Theorem~V.4.1]{massey}, the covering maps $(p_n)_*$ induce injective homomorphisms
$\wilde{(p_n)}_*$ of fundamental groups. Theorems II.2.2~and~II.2.3 of \cite{spanier} imply that
the covering maps $(p_n)_* : X_{\Lambda_{n+1}} \to X_{\Lambda_n}$ are fibrations with unique path
lifting, so \cite[Corollary~VII.2.11]{spanier} implies that the maps $\pi_2(X_{\Lambda_{n+1}}) \to
\pi_2(X_{\Lambda_n})$ induced by the $(p_n)_*$ are isomorphisms. It therefore follows from
\cite[Proposition~4.67]{Hatcher} that
\[
\pi_1(X_{\varprojlim (\Lambda_n, p_n)}) \cong \varprojlim(\pi_1(X_{\Lambda_n}), \wilde{(p_n)}_*).
\]
That $\varprojlim(\pi_1(X_{\Lambda_n}), \wilde{(p_n)}_*) \cong \varprojlim(\pi_1(\Lambda_n),
(p_n)_*)$ follows from Lemma~\ref{functor group}.
\end{proof}

\section{Crossed products and mapping tori}\label{sec:crossed products}

Let $\Lambda$ be a $k$-graph, and $\alpha : \ZZ^l \to \Aut(\Lambda)$ an action by automorphisms.
Recall that the crossed-product $k$-graph $\Lambda \times_\alpha \ZZ^l$ is equal as a set to $\Lambda
\times \NN^l$ and has operations $r(\lambda,m) = (r(\lambda), 0)$, $s(\lambda, m) =
(\alpha_{-m}(s(\lambda)), 0)$ and $(\lambda,m)(\mu,n) = (\lambda\alpha_m(\mu), m + n)$.

Now let $X$ be a topological space, and $\sigma$ an action of $\ZZ^l$ on $X$ by homeomorphisms.
Then there is an action $\sigma \times \lt$ of $\ZZ^l$ on $X \times \RR^l$ given by $(\sigma \times
\lt)_m(x, t) = (\sigma_m(x), m + t)$. The \emph{mapping torus} of $\sigma$ is the orbit space
\[
M(\sigma) = (X \times \RR^l)/(\sigma \times \lt).
\]
We denote the equivalence class of $(x, t) \in X \times \RR^l$ in the mapping torus by
$[x,t]_{M(\sigma)}$, where the subscript is to distinguish such classes from elements of
topological realizations $X_\Lambda$ of $k$-graphs $\Lambda$, or simply by $[x,t]$ if there is no
possibility of confusion.

In the following Lemma, we identify $\RR^{k+l}$ with $\RR^k \times \RR^l$ in the standard way.

\begin{lem}\label{lem:mapping-torus}
Let $\Lambda$ be a $k$-graph and $\alpha$ an action of $\ZZ^l$ on $\Lambda$. Let $\wilde{\alpha}$
be the induced action of $\ZZ^l$ on $X_\Lambda$ obtained from functoriality of topological
realization. Then there is a homeomorphism $\varphi : M(\wilde{\alpha}) \cong X_{\Lambda
\times_\alpha \ZZ^l}$ determined by
\begin{equation}\label{eq:phi formula}
\varphi\big(\big[[\lambda,s], t\big]_{M(\wilde{\alpha})}\big) = [(\lambda, \ceil[t]), (s,t)]
\end{equation}
whenever $t \ge 0$.
\end{lem}
\begin{proof}
For any $\big[[\lambda,s],t\big]_{M(\wilde{\alpha})} \in M(\wilde{\alpha})$ and $p \in \NN^l$ such
that $p + t \ge 0$, we have
\[
\big[[\lambda,s], t\big]_{M(\wilde{\alpha})}
    = \big[(\wilde{\alpha} \times \lt)_p([\lambda,s], t)\big]_{M(\wilde{\alpha})}
    = \big[[\alpha(\lambda),s], t + p\big]_{M(\wilde{\alpha})},
\]
so each point in $X_\Lambda$ has the form $\big[[\lambda,s], t\big]_{M(\wilde{\alpha})}$ where $t >
0$. To see that $\varphi$ is well-defined, suppose $[[\lambda,s],t]_{M(\wilde{\alpha})} =
[[\lambda',s'],t']_{M(\wilde{\alpha})}$ with $t, t' \ge 0$. Then $t - \floor[t] = t' - \floor[t']$
and
\[
[\alpha_{\floor[t]}(\lambda), s]
    = (\wilde{\alpha} \times \lt)_{\floor[t]}(\lambda,s)
    = (\wilde{\alpha} \times \lt)_{\floor[t']}(\lambda',s')
    = [\alpha_{\floor[t']}(\lambda'), s'].
\]
Hence $s - \floor[s] = s' - \floor[s']$, and
\begin{equation}\label{eq:product action image}
\big(\alpha_{\floor[t]}(\lambda)\big)(\floor[s], \ceil[s]) = \big(\alpha_{\floor[t']}(\lambda')\big)(\floor[s'], \ceil[s']).
\end{equation}
We have $\big(\alpha_{\floor[t]}(\lambda), 0\big)((\floor[s],0), (\ceil[s],0)) =
(\lambda,t)((\floor[s],\floor[t]), (\ceil[s],\floor[t]))$ by definition of composition in $\Lambda
\times_\alpha \ZZ^l$. Substituting this and the symmetric equality for $\lambda'$
into~\eqref{eq:product action image} gives
\[
(\lambda,t)((\floor[s],\floor[t]), (\ceil[s],\floor[t])) = (\lambda',t')((\floor[s'],\floor[t']), (\ceil[s'],\floor[t'])).
\]
Multiplying both sides on the right by $(s(\lambda), \ceil[t] - \floor[t]) =
(s(\lambda'), \ceil[t'] - \floor[t'])$, we obtain
\[
(\lambda,t)((\floor[s],\floor[t]), (\ceil[s],\ceil[t])) = (\lambda',t')((\floor[s'],\floor[t']), (\ceil[s'],\ceil[t'])).
\]
Since $s - \floor[s] = s' - \floor[s']$ and $t - \floor[t] = t' - \floor[t']$, we have $(s,t) -
\floor[(s,t)] = (s',t') - \floor[(s',t')]$, whence $[(\lambda, \ceil[t]), (s,t)] = [(\lambda',
\ceil[t']), (s',t')]$ as required. In particular, given any
$\big[[\lambda,s],t\big]_{M(\wilde{\alpha})}$, any two representatives of this element with
positive $t$-value have the same image under the formula~\eqref{eq:phi formula}. So there is a
well-defined map $\varphi : M(\wilde{\alpha}) \to X_{\Lambda \times_\alpha \ZZ^l}$
satisfying~\eqref{eq:phi formula}. The map $\varphi$ is clearly surjective. To see that it is
injective, just reverse the reasoning of the preceding paragraph: if $\varphi\big(\big[[\lambda,s],
t\big]_{M(\wilde{\alpha})}\big) = \varphi\big(\big[[\lambda',s'],
t'\big]_{M(\wilde{\alpha})}\big)$, then
\begin{align*}
(s,t) - \floor[(s,t)]
    &= (s',t') - \floor[(s',t')],\quad\text{ and } \\
\big(\alpha_{\floor[t]}(\lambda)\big)(\floor[s], \ceil[s])
    &= \big(\alpha_{\floor[t']}(\lambda')\big)(\floor[s'], \ceil[s']),
\end{align*}
whence $\big[[\lambda,s], t\big]_{M(\wilde{\alpha})} = \big[[\lambda',s'],
t'\big]_{M(\wilde{\alpha})}$.

To see that $\varphi$ is continuous, observe that if $d(\lambda) \le \One{k}$ and $n \le \One{l}$,
then the inverse image of the closed cube $\overline{Q_{(\lambda,n)}}$ under $\varphi$ is
\[
\big\{\big[[\lambda,s], t\big]_{M(\wilde{\alpha})} : 0 \le s \le \One{k}, 0 \le t \le \One{l}\big\}
    = (\overline{Q_\lambda} \times [0, \ceil[t]])/(\wilde{\alpha} \times \lt),
\]
which is closed.

Finally, to prove that $\varphi$ is a homeomorphism, it remains to verify that the inverse
$\varphi\inv: X_{\Lambda\times_\alpha \ZZ^l} \to M(\wilde{\alpha})$ is also continuous. Fix
$\lambda \in \Lambda$ with $d(\lambda) \le \One{k}$ and fix $n \le \One{l}$. Then the restriction
of $\varphi\inv$ to $\overline{Q_{(\lambda,n)}}$ is a homeomorphism onto $\big[[\lambda',s'],
t'\big]_{M(\wilde{\alpha})}$ and is therefore continuous.  Since $X_{\Lambda\times_\alpha \ZZ^l}$
is endowed with the weak topology determined by closed cubes, this proves that
$\varphi\inv$ is continuous as required.
\end{proof}

Recall that a \emph{deck transformation} of a covering map $p : X \to Y$ is a homeomorphism $g$ of $X$
such that $p \circ g = p$. The deck transformations of $p$ form a group $D(p)$.

\begin{prop}\label{prp:extension}
Let $X$ be a connected $CW$-complex and let $\sigma$ be an action of $\ZZ^l$ on $X$ by
homeomorphisms. Fix $x_0 \in X$. Let $i_X : X \to M(\sigma)$ denote the embedding given by $i_X(x)
= [x,0]$. Then $i_X$ induces an injection $(i_X)_*:  \pi_1(X,x_0) \to \pi_1(M(\sigma),[x_0,0])$
such that $(i_X)_*(\pi_1(X,x_0))$ is a normal subgroup of $\pi_1(M(\sigma),[x_0,0])$. Moreover,
\[
\pi_1(M(\sigma),[x_0,0])/(i_X)_*( \pi_1(X,x_0)) \cong \ZZ^l.
\]
\end{prop}
\begin{proof}
Functoriality of $\pi_1$ yields a homomorphism $(i_X)_* : \pi_1(X,x_0) \to
\pi_1(M(\sigma),[x_0,0])$.

Consider the space $X \times \RR^l$. Let $\sigma \times \lt$ be the action of $\ZZ^l$ determined by
$(\sigma \times \lt)_n(x,t) = (\sigma^n(x), t + n)$. Then $M(\sigma)$ is by definition the orbit
space of this action. For $(x,t) \in X \times \RR^l$, any open neighborhood $N$ of the form $U
\times B(t; \frac{1}{3})$ of $(x,t)$ has the property that $(\sigma \times \lt)_m(N) \cap (\sigma
\times \lt)_n(N) = \emptyset$ for distinct $m,n \in \ZZ^l$. So $\sigma \times \lt$ satisfies
condition~(*) of \cite[Page~72]{Hatcher}. Hence \cite[Proposition~1.40]{Hatcher} implies first that
the quotient map $q : X \times \RR^l \to M(\sigma)$ is a regular covering whose deck-transformation
group $D(q)$ is isomorphic to $\ZZ^l$, second that $q_*(\pi_1(X \times \RR^l, (x_0, 0)))$ is a
normal subgroup of $\pi_1(M(\sigma))$, and third that the quotient is isomorphic to $D(q)$. So we
just need to see that $(i_X)_*$ is an injection and its image coincides with $q_*(\pi_1(X \times
\RR^l, (x_0, 0)))$. For this, observe that since $\pi_1(\RR^l)$ is trivial,
\cite[Theorem~II.7.1]{massey} implies that $j_X : x \mapsto (x,0)$ from $X$ to $X\times\RR^l$
induces an isomorphism $(j_X)_* : \pi_1(X, x_0) \to \pi_1(X \times \RR^l, (x_0, 0))$. We have
$(i_X)_* = (q \circ j_X)_* = q_* \circ (j_X)_*$. Since \cite[Theorem~V.4.1]{massey} implies that
$q_*$ is injective, it follows that $(i_X)_*$ is injective with the same image as $q_*$, as
required.
\end{proof}

The following Corollary is an immediate consequence of \propref{prp:extension}, the functoriality
of the fundamental group and \corref{isomorphism}.

\begin{cor}
Let $\Lambda$ be a connected $k$-graph, let $u \in \Lambda^0$
and let $\alpha$ be an action of $\ZZ^l$ on $\Lambda$.  Then there is an extension
\[
1 \to \pi_1(\Lambda, u) \xrightarrow{(i_\Lambda)_*}
\pi_1(\Lambda \times_\alpha \ZZ^l, (u,0)) \to \ZZ^l \to 0,
\]
where $i_\Lambda: \Lambda \to \Lambda \times_\alpha \ZZ^l$ is the canonical embedding.
\end{cor}

\end{document}